\documentclass[11pt,a4paper]{amsart} 

\usepackage[margin=2.9cm]{geometry}	
\usepackage{amsmath}	
\usepackage{amsthm}	
\usepackage{amssymb}	
\usepackage{extarrows}	
\usepackage{mathtools}	
\usepackage{thmtools}	
\usepackage{amsfonts}	
\usepackage[scr=boondoxo]{mathalfa}	
\usepackage{fontawesome}	
\usepackage{bm}	
\usepackage[shortlabels]{enumitem}	
\usepackage{graphicx}	
\usepackage[font={small}]{caption}	
\usepackage{subcaption}	
\usepackage[dvipsnames]{xcolor}	
\usepackage{comment}	
\usepackage{etoolbox}	

\usepackage{lmodern}	
\usepackage{upgreek}	
\usepackage[doi=true,isbn=true,giveninits=true,maxbibnames=99,backend=biber,url=true]{biblatex}	
\usepackage{hyperref}	
\usepackage{xurl}	
\usepackage[capitalise]{cleveref}	
\hypersetup{colorlinks=true,linktoc=section,urlcolor={Aquamarine!85!black},
	citecolor={YellowOrange!85!black},linkcolor={violet}}										
\makeatletter	
\AtBeginDocument{
\patchcmd\label@noarg{\edef\@tempb}{\protected@edef\@tempb}{}{}}	
\makeatother	
\renewbibmacro*{in:}{
	\iffieldundef{journaltitle}																	
	{}																							
	{\printtext{\bibstring{in}\intitlepunct}}													
}

\theoremstyle{plain}	
\newtheorem{thm}{Theorem}[section]	
\crefname{thm}{Theorem}{Theorems}	
\newtheorem{lem}[thm]{Lemma}	
\crefname{lem}{Lemma}{Lemmata}	
\newtheorem{prop}[thm]{Proposition}	
\newtheorem{cor}[thm]{Corollary}	
\theoremstyle{definition}	
\newtheorem{dfn}[thm]{Definition}	
\newtheorem{eg}[thm]{Example}	
\crefname{eg}{Example}{Examples}	
\theoremstyle{plain}	
\theoremstyle{plain}	
\newtheorem*{thm*}{Theorem}	
\newtheorem*{lem*}{Lemma}	
\newtheorem*{prop*}{Proposition}	
\newtheorem*{cor*}{Corollary}	
\theoremstyle{definition}	
\newtheorem*{dfn*}{Definition}	
\newtheorem*{eg*}{Example}	
\newtheorem*{not*}{Notations}	
\theoremstyle{remark}	
\newtheorem*{rem*}{Remark}	
\newtheorem*{que*}{Question}	
\theoremstyle{plain}	
\newtheorem{thmabc}{Theorem}	
\newtheoremstyle{named}{}{}{\itshape}{}{\bfseries}{.}{.5em}{\thmnote{#3}}	
\theoremstyle{named}	
\numberwithin{equation}{section}	
\patchcmd{\subsection}{-.5em}{1sp}{}{}

\makeatletter	
\def\hrulefill{\noindent\leavevmode\leaders\hrule\hfill\kern\z@}	
\makeatother	
\renewcommand{\mod}{\operatorname{mod}}
\newcommand{\such}{\,:\,}	
\DeclareMathOperator{\dist}{dist}	
\DeclareMathOperator{\id}{id}	
\DeclareMathOperator{\Arg}{Arg}	
\newcommand{\inline}[1]{\quad\text{#1}\quad}	
\newcommand{\afterline}[1]{\qquad\text{#1~}}	
\newcommand{\inmath}[1]{\text{ #1~}}			
\ExplSyntaxOn	
\NewDocumentCommand{\definealphabet}{mmmm}{
	\int_step_inline:nnn { `#3 } { `#4 }{
		\cs_new_protected:cpx { #1 \char_generate:nn { ##1 }{ 11 } }{
			\exp_not:N #2 { \char_generate:nn { ##1 } { 11 } }									
		}																						
	}																							
}																								
\ExplSyntaxOff																					
\definealphabet{mb}{\mathbb}{A}{Z}	
\definealphabet{mc}{\mathcal}{A}{Z}	
\definealphabet{ms}{\mathscr}{A}{Z}	
\DeclareMathOperator{\jul}{\mathscr{J}}	
\DeclareMathOperator{\fat}{\mathscr{F}}	
\newcommand{\CC}{{\hat{\mathbb{C}}}}
\newcommand{\seq}[1]{{\bm{{#1}}}}	
\DeclareMathOperator{\shift}{\uptau}

\title{Julia Sets for Sequences of Monomials}
\author[L.~Pathirana]{Lubashan Pathirana}
\address{K{\o}benhavns Universitet}
\email{lpk@math.ku.dk}
\author[C.~L.~Petersen]{Carsten Lunde Petersen}
\address{K{\o}benhavns Universitet}
\email{lunde@math.ku.dk}
\author[D.~Vardakis]{Dimitris Vardakis}
\address{Julius-Maximilians-Universit{\"a}t W{\"u}rzburg}
\email{jimvardakis@gmail.com}

\begin{document}


\begin{abstract}
	We study Julia sets arising from non-autonomous iteration of monomials on the Riemann sphere.
	We give a precise description of the Julia set associated to an arbitrary sequence of monomials in terms of their coefficients and degrees.
	Combined with the basic invariance relation for non-autonomous Julia sets,
	this yields a plethora of striking examples in polynomial non-autonomous dynamics:
	finite Julia sets of arbitrary cardinality,
	non-trivial Julia sets with non-empty interior,
	Julia sets that are perfect but not uniformly perfect,
	and Julia sets with empty interior but positive area.
	We also describe a connection between autonomous and non-autonomous Julia sets of monomial sequences through Kuratowski limits.
\end{abstract}




\maketitle



\section{Introduction}
\label{sec:intro}

Let~$f:\CC\to\CC$ be a rational map on the Riemann sphere~$\CC:=\mbC\cup\{\infty\}$,
and consider the family of its iterates~$\mcF_f = \{f^{\circ n} \such n\in\mbN\}$.
It is a standard problem to describe on which set this family is normal;
the maximal such open set is known as the Fatou set of~$f$,~$\fat(f)$ and its complement as the Julia set of~$f$,~$\jul(f)$.
In this \emph{autonomous} setting we know that Julia sets of maps of degree at least~$2$ are non-empty perfect sets.
In principle, much of the dynamics of the family~$\mcF_f$ is reflected in the geometry of~$\jul(f)$ and~$\fat(f)$;
see, for example, \cite[\S\S2--3]{Milnor2006} or \cite[\S2.1]{Beardon1991}.

In this text, we consider the corresponding \emph{non-autonomous} problem:
Given a sequence~$\seq{f}=(f_1,f_2,\dots)$ of rational maps from~$\CC$ to~$\CC$,
we consider the non-autonomous iterates
\[
	F^{(n)}_\seq{f} := f_n\circ\dots\circ f_1 \afterline{for} n\in\mbN,
\]
and we define the \emph{Fatou set}~$\fat(\seq{f})$ of the sequence~$\seq{f}$ as maximal the set of points on which the family~$\{F^{(n)}_\seq{f}\}_{n\ge1}$ is normal.
We call its complement~$\jul(\seq{f}) := \CC\setminus \fat(\seq{f})$ the \emph{Julia set} of the sequence~$\seq{f}$.
Even for polynomial sequences, the non-autonomous setting is substantially more flexible than the autonomous one, and familiar conclusions such as perfectness or emptiness of interior may fail.
The objective of this paper is to provide simple examples 
based on ``monomials'' of this flexibility.

For sequences of monomials, the dynamics are radial in the sense that the transformation~$z\mapsto z e^{i\theta}$ leaves the Julia and Fatou sets unchanged.
This yields a particularly transparent class of examples while still exhibiting phenomena impossible in the autonomous setting.
In fact, our main theorem below gives a complete and precise description of Julia sets for monomial sequences.

\begin{restatable*}[]{thm}{monomialjulia}
\label{thm:monomial_julia_accumulation}
	Consider a sequence of monomials~$\seq{p}=(p_n)_{n\ge1}$ with~$p_n(z)=a_n z^{d_n}$,~$a_n\in\mbC\setminus\{0\}$ and~$d_n\ge1$,
	and assume that~$d_n\ge2$ for infinitely many~$n$.
	Set
	\[
		D_n := \prod\nolimits_{i=1}^n d_i,
		\quad A_n := \prod\nolimits_{i=1}^n a_i^{D_n/D_i}
		\inline{and}
		s_n := \frac{1}{D_n} \log|A_n| = \sum\nolimits_{i=1}^n \frac{1}{D_i} \log|a_i|,
	\]
	and let~$S$ be the set of accumulation points of the sequence~$(s_n)_{n\geq1}$ in~$\overline{\mbR} := \mbR\cup\{\pm\infty\}$.
	Then,
	\[
		\jul(\seq{p})
		=
		\bigcup\nolimits_{s\in S}\{z\in\CC \such |z|=e^{-s}\}
	\]
	with the conventions~$e^{-\infty}=0$ and~$e^{-(-\infty)}=+\infty$.
\end{restatable*}

Thus the geometry of~$\jul(\seq{p})$ is determined exactly by the accumulation points of the real sequence~$(s_n)_{n\geq1}$.
Here are a few examples.
By prescribing a non-empty closed set~$S\subseteq\overline{\mbR}$, one obtains a large family of radial Julia sets:
Cantor sets of circles, including examples with positive planar area and empty interior, perfect Julia sets that are not uniformly perfect, among other examples.
We provide a few such examples and constructions at \cref{sec:examples and contructions}.

We also give two complementary descriptions of the monomial Julia set at \cref{sec:orbital description,sec:Kuratowski}.
One is dynamical:~$\jul(\seq{p})$ can be recovered as the derived set of the level-by-level preimages of an appropriate sequence of points under the non-autonomous iterative compositions~$F^{(n)}_{\seq{p}}$.
The other is asymptotic:~$\jul(\seq{p})$ is the Kuratowski limit superior of the autonomous Julia sets of the individual non-autonomous iterative compositions~$F^{(n)}_{\seq{p}}$ for each~$n$,
i.e.~of the autonomous dynamical systems~$\{\big(F^{(n)}_{\seq{p}}\big)^{\circ k}\}_{k\geq1}$.


\section{Preliminaries and notation}
\label{sec:prelim and notations}


We follow standard notation from complex dynamics and normal family theory; see, for example, \cite[\S2]{Milnor2006} or \cite[\S2.1]{Beardon1991}.
We write~$\mbC$ for the complex plane,~$\CC := \mbC\cup\{\infty\}$ for the Riemann sphere,~$\overline{\mbR} := \mbR\cup\{\pm\infty\}$ for the extended real line,
and
\[
	D(z,r) := \{w\in\mbC \such |w-z|<r\},
	\inline{and}
	C(z,r) := \{w\in\mbC \such |w-z|=r\}.
\]
for the open disk and the circle, respectively, in~$\mbC$ centred at~$z\in\mbC$ with radius~$r>0$.
More generally, if~$d$ is a metric on a space~$X$, we write
\[
	D_d(x,r) := \{y\in X \such d(y,x)<r\}
	\inline{and}
	C_d(x,r) := \{y\in X \such d(y,x)=r\}
\]
for the open~$d$-ball and the circle, respectively, of radius~$r>0$ centred at~$x \in X$.
When a sequence~$(x_n)_{n\in\mbN}$ of points in~$X$ converges to~$x \in X$ in the metric~$d$ as~$n\to+\infty$ we write~$x_n \to_d x$.
If no metric is specified, it is understood that the metric used is the Euclidean norm on the respective space.
When we have a sequence of points~$(a_n)_{n\in\mbN}$ in~$\mbR$ converging to a point~$a\in\overline{\mbR}$ we mean that either~$a\in\mbR$ and~$a_n \to a$ or~$a_n\to\pm\infty$ in the Euclidean norm.
We also set~$\log0 := -\infty$,~$\log(+\infty) := +\infty$,~$e^{-\infty} := 0$,~$e^{+\infty} := +\infty$, and~$|\infty| := +\infty$.

Throughout, a \emph{domain} means a connected open set.

We equip~$\CC$ with the standard spherical metric~$\sigma$, see for example \cite[\S2.1]{Beardon1991}.
All convergence and normality statements below are understood with respect to~$\sigma$.
Since
\[
	\sigma(z,w)
	\leq \frac{\pi}{2} \frac{2|z-w|}{\sqrt{1+|z|^2}\sqrt{1+|w|^2}}
	\leq \pi |z-w| \afterline{for any} z,w\in\mbC,
\]
in many of our proofs below it suffices to use the Euclidean metric in~$\mbC$ whenever the point at infinity need not be concerned.
When~$w=\infty$, we instead have
\[
	\sigma(z,\infty)
	=\sigma(z^{-1},0)
	\leq \frac{\pi}{2} \frac{2}{\sqrt{1+|z|^2}} \afterline{for any} z\in\mbC\setminus\{0\},
\]
and also~$\sigma(0,\infty) = \pi$ and~$\sigma(\infty,\infty)=0$.
Additionally, the following inequality holds
\begin{equation}
\label{eq:chordal less than spherical}
	\frac{2|z-w|}{\sqrt{1+|z|^2}\sqrt{1+|w|^2}} \leq \sigma(z,w) \afterline{for all} z,w\in\mbC.
\end{equation}
The quantity on the left-hand side of \eqref{eq:chordal less than spherical} is known as the \emph{chordal metric}.

We use the letters~$m$,~$n$,~$k$,~$i$ and~$j$ to denote natural, or sometimes integer, numbers.

Boldface characters and indexed characters inside parentheses are used to denote sequences.
Calligraphic characters are used to denote families of maps.


\subsection{Normal families}
\label{sec:normal fam}

We denote the space of all holomorphic mappings from a domain~$U\subseteq\CC$ to~$\CC$ by~$\mcH(U,\CC)$.
We denote the space of all holomorphic \emph{mappings} from~$\CC$ to~$\CC$ by~$\mcH(\CC)$;
this consists of rational maps and the constant map~$\infty$.

Let~$U\subseteq\CC$ be a domain and consider the holomorphic mapping~$f$ and a sequence of holomorphic mappings~$(f_n)_{n\geq1}$ in~$\mcH(U,\CC)$.
We say that~$(f_n)$ converges \emph{locally uniformly on~$U$ to~$f$} if for every compact~$K\subseteq U$,
\[
	\sup\nolimits_{z\in K}\sigma\big(f_n(z),f(z)\big) \to 0
	\afterline{as} n\to+\infty.
\]

A family~$\mcF$ of holomorphic mappings in~$\mcH(U,\CC)$ is called \emph{normal on~$U\subseteq\CC$} if every sequence in~$\mcF$ has a subsequence that converges locally uniformly on~$U$ to a holomorphic mapping from~$U$ to~$\CC$;
a family~$\mcF$ of holomorphic functions from~$U\subseteq\mbC$ to~$\mbC$ is called \emph{normal on~$U$} if every sequence in~$\mcF$ has a subsequence that converges on~$U$ to a holomorphic function from~$U$ to~$\mbC$.

A family~$\mcF$ of holomorphic mappings in~$\mcH(U,\CC)$ is called \emph{equicontinuous on~$U$ with respect to the spherical metric} or \emph{spherically equicontinuous on~$U$} when for all~$w \in U$ for any~$\varepsilon>0$ there exists a~$\delta>0$ such that for every~$f\in\mcF$
\[
	\sigma(f(z),f(w))<\varepsilon \afterline{whenever} z \in U \inmath{satisfies} \sigma(z,w)<\delta.
\]

A family of holomorphic functions~$\mcF$ is \emph{locally bounded} on a domain~$U\subseteq\mbC$ if for each compact subset~$K\subseteq U$, there is some~$M_K<+\infty$ such that~$|f(z)|\le M_K$ for all~$z \in K$ and for all~$f\in \mcF$;
see \cite[Chapter~VII; Definition~2.7 \& Lemma~2.8]{Conway1978}.

We shall use the following two standard criteria to determine the normality of a given family;
see \cite[Theorem~3.3.2]{Beardon1991} and \cite[Chapter~VII; Theorem~2.9]{Conway1978}, respectively.

\begin{thmabc}[Montel's Theorem]
\label{thm:Montel theorem spherical}
	A family~$\mcF$ of holomorphic mappings from a domain~$U\subseteq\CC$ to~$\CC$ is normal on~$U$ if, and only if, it is equicontinuous on~$U$ with respect to the spherical metric.
\end{thmabc}

\begin{thmabc}[Montel's Theorem for holomorphic functions on~$\mbC$]
\label{thm:Montel for holomorphic}
	A family~$\mcF$ of holomorphic functions in a domain~$U\subseteq\mbC$ is normal on~$U$ if, and only if, it is locally bounded on~$U$.
\end{thmabc}

For a meromorphic function~$f:U\to\mbC$ on a plane domain~$U\subseteq\mbC$, its \emph{spherical derivative,~$f^\#$,} at~$z \in U$ is defined as
\begin{gather*}
	f^\#(z) := \frac{2|f'(z)|}{1+|f(z)|^2}
\inline{or}
	f^\#(z) := \lim_{\zeta \to z} \frac{2|f'(\zeta)|}{1+|f(\zeta)|^2}
\end{gather*}
depending on whether~$z$ is a pole of~$f$ or not.
This is a well-defined and continuous (real-valued) function on~$U$ \cite[Chapter~VII; Definition~3.7]{Conway1978},
and it facilitates another criterion for normality \cite[Theorem~3.8]{Conway1978}.

\begin{thmabc}[Marty's Criterion]
\label{thm:Marty}
	Let~$U\subseteq\mbC$ be a domain and let~$\mcF$ be a family of meromorphic functions on~$U$.
	Then,~$\mcF$ is normal on~$U$ if, and only if, the family of spherical derivatives~$\mcF^\# := \{f^\# \such f\in\mcF\}$ is locally bounded on~$U$.
\end{thmabc}


\subsection{Non-autonomous compositions}


Let~$A$ be any non-empty set.
We denote by~$\shift$ the \emph{left shift} on sequences, given by
\[
	\shift^k(a_1,a_2,a_3,\dots)=(a_{k+1},a_{k+2},a_{k+3},\dots)
\]
for any sequence of elements~$(a_n)_{n\geq1}\in A^\mbN$ and any~$k\in\mbN$.
For~$k=0$, we define~$\shift^0 = \id_{A^\mbN}$.

Let~$\seq{f} := (f_n)_{n\geq1} = (f_1,f_2,\dots) \in \mcH(\CC)^\mbN$ be a sequence of rational maps on~$\CC$.
For~$n\ge1$, we define the non-autonomous iterates
\[
	F^{(n)}_\seq{f} := f_n\circ\dots\circ f_1
	\inline{and also set}
	F^{(0)}_\seq{f} := \id_{\CC};
\]
these satisfy that~$F^{(n)}_{\shift\seq{f}}\circ f_1 = f_{n+1} \circ F^{(n)}_\seq{f}= F^{(n+1)}_\seq{f}$
along with the following \emph{cocycle property}:
\[
	F^{(n)}_{\shift^k\seq{f}}\circ F^{(k)}_\seq{f}=F^{(n+k)}_\seq{f} \afterline{for all} n,k\ge0.
\]

\begin{dfn}[The Fatou and Julia sets of a sequence]
	Let~$\seq{f}=(f_n)_{n\ge1}$ be a sequence of rational maps on~$\CC$, and consider the family~$\mcF_\seq{f} := \{F^{(n)}_\seq{f} \such n\ge1\}$.
	The maximal open set in~$\CC$ where~$\mcF_\seq{f}$ is normal we denote by~$\fat(\seq{f})$ and call it the \emph{Fatou set} of the sequence~$\seq{f}$.
	The \emph{Julia set} of the sequence~$\seq{f}$ is the complement of the Fatou set~$\jul(\seq{f}) := \CC \setminus \fat(\seq{f})$.
\end{dfn}

By definition,~$\fat(\seq{f})$ is open and~$\jul(\seq{f})$ is closed.

Sometimes we will refer to these sets as the \emph{non-autonomous} Julia and Fatou sets.
We call the ``classic'' Julia and Fatou sets produced by iterates of the same mapping as the \emph{autonomous} Julia and Fatou sets;
for these we may respectively write~``$\jul^{\mathrm{aut}}$'' and~``$\fat^{\mathrm{aut}}$''.


\subsection{Sequences of Monomials}
\label{sec:monomials definitions}


Consider a sequence of monomials~$\seq{p}=(p_n)_{n\ge1} = (p_1,p_2,p_3,\dots)$ in~$\CC$ with~$p_n(z)=a_n z^{d_n}$,~$a_n\in\mbC\setminus\{0\}$ and~$d_n\ge1$.
We introduce the following notation,
which will be used throughout this text:
\[
	D_n := \prod\nolimits_{j=1}^n d_j,
	\inline{and}
	A_n := \prod\nolimits_{i=1}^n a_i^{D_n/D_i}
	\inline{so that}
	F^{(n)}_{\seq{p}}(z)=A_n z^{D_n}.
\]
Further, we write~$S\subseteq\overline\mbR$ for the set of accumulation points of~$(s_n)_{n\geq1}$ in~$\overline{\mbR}$,
we set
\begin{equation}
\label{eq:s_n}
	s_n := \frac{1}{D_n}\log |A_n| = \sum\nolimits_{i=1}^n \frac{1}{D_i}\log|a_i|,
\end{equation}
It is worth noting that the set~$S$ cannot be empty as~$\overline{\mbR}$ is a compact set.

We might refer to a ``sequence of monomials'' as a \emph{monomial sequence} and to a ``sequence of polynomials'' as a \emph{polynomial sequence}.


\section{General invariance and non-emptiness}
\label{sec:basic_invariance}


Here, we present some general fundamental properties of the Julia and Fatou sets in the non-autonomous setting.

First, we present the equivalent of the invariance property for non-autonomous Fatou and Julia sets.
Versions of this are already known,
for example in \cite[Remark 3.1]{Sim2013}, or for bounded polynomial sequences in \cite[Theorem~1.1]{ComerfordSumi-ThickThin};
we provide our proof here for completeness.

\begin{thm}[Invariance of Fatou and Julia sets]
\label{thm:invariance deterministic}
	Let~$\seq{f} \in \mcH(\CC)^\mbN$ be a sequence of non-constant holomorphic mappings.
	Then, for the Julia set~$\jul(\seq{f})$ it holds that
	\begin{gather*}
		f_1(\jul(\seq{f})) = \jul(\shift \seq{f})
		\inline{and}
		f_1^{-1} \jul(\shift \seq{f}) = \jul(\seq{f}),
	\intertext{and consequently, for any~$n\in\mbN$}
		F^{(n)}_\seq{f}(\jul(\seq{f})) = \jul(\shift^n \seq{f})
		\inline{and}
		\big(F^{(n)}_\seq{f}\big)^{-1} \jul(\shift^n \seq{f}) = \jul(\seq{f}).
	\end{gather*}
	The same exact equalities also hold with the Julia set~$\jul(\seq{f})$ replaced by the Fatou set~$\fat(\seq{f})$.
\end{thm}

\begin{proof}
	It is enough to show that~$f_1^{-1} \fat(\shift \seq{f}) = \fat(\seq{f})$ holds,
	for any given a sequence~$\seq{f}=(f_1,f_2,\dots)$ where each~$f_n$ has degree at least~$1$,
	in which case each~$f_n$ is surjective.		
	Thus, if~$f_1^{-1} \fat(\shift \seq{f}) = \fat(\seq{f})$, then~$\fat(\shift \seq{f}) = f_1 \big( f_1^{-1} \fat(\shift \seq{f}) \big) = f_1(\fat(\seq{f}))$.
	
	Next, consider the sequence~$\seq{g} = \shift^{n-1} \seq{f}$ for~$n\in \mbN$.
	(The operator~$\shift^0$ is merely the identity.)
	If~$f_1^{-1} \fat(\shift \seq{f}) = \fat(\seq{f})$ is true for any~$\seq{f}\in\mcH(\CC)^\mbN$, it must also be true for~$\seq{g}$, that is,~$g_1^{-1}\fat(\shift \seq{g}) = \fat(\seq{g})$, in which case we obtain that 
	$f_n^{-1}\fat(\shift^n \seq{f}) = \fat(\shift^{n-1} \seq{f})$.
	Iterating this argument, we get~$\big(F^{(n)}_\seq{f}\big)^{-1} \fat(\shift^n \seq{f}) = f_1^{-1} \dots f_n^{-1} \fat(\shift^n \seq{f}) = \fat(\seq{f})$ for any~$n\in\mbN$.
	As above,~$\big(F^{(n)}_\seq{f}\big)^{-1} \fat(\shift^n \seq{f}) = \fat(\seq{f})$ implies that~$\fat(\shift^n \seq{f}) = F^{(n)}_\seq{f}(\fat(\seq{f}))$,
	since~$F^{(n)}_\seq{f}$ is surjective as a composition of surjective maps.
	
	Hence, showing~$f_1^{-1} \fat(\shift \seq{f}) = \fat(\seq{f})$ for all sequences~$\seq{f}\in\mcH(\CC)^\mbN$ of non-constant holomorphic mappings suffices to conclude all the requested properties for the Fatou set.
	We proceed to do this right below.
	The corresponding results for~$\jul(\seq{f})$ are true immediately as the Fatou and Julia sets are complementary, and the mappings~$f_n$ are all surjective.
	
	We start by showing the inclusion~$f_1^{-1}\fat(\shift \seq{f}) \subseteq \fat(\seq{f})$.
	Let~$z_0\in f_1^{-1}\fat(\shift \seq{f})$ so that~$f_1(z_0) \in \fat(\shift \seq{f})$.
	As the Fatou set is open, there is a neighbourhood~$U$ of~$f_1(z_0)$ such that the family
	\[\mcF_{\shift \seq{f}} = \{f_{n+1} \circ\dots\circ f_2 \such n\in\mbN\}\]
	is spherically equicontinuous at each point of~$U$.
	Now, consider the neighbourhood~$V := f_1^{-1}(U)$ of~$z_0$,
	---$V$ is open as~$f_1$ is continuous---
	and fix an~$\varepsilon>0$.
	Since~$f_1(z_0) \in U$ and~$\mcF_{\shift \seq{f}}$ is spherically equicontinuous on~$U$, there exists a~$\delta_1>0$ such that for all~$n \in \mbN$ and all~$w \in U$ with~$\sigma(w,f_1(z_0))<\delta_1$ it holds that
	\[
		\sigma(f_{n+1} \circ\dots\circ f_2(w), f_{n+1} \circ\dots\circ f_2(f_1(z_0)))<\varepsilon.
	\]
	However, by the continuity of~$f_1$ on~$V$ for every~$\delta>0$ there exists a~$\rho>0$ such that for all~$z \in V$ with~$\sigma(z,z_0)<\rho$ one has that~$\sigma(f_1(z),f_1(z_0))<\delta$.
	Therefore, since~$f_1(z)\in U$ when~$z \in V$, setting~$\delta = \delta_1$ and combining the two inequalities from above gives that there exists a~$\rho>0$ such that for all~$n \in \mbN$ and all~$z \in V$ with~$\sigma(z,z_0)<\rho$ it holds that
	\[
		\sigma(F^{(n+1)}_\seq{f}(z), F^{(n+1)}_\seq{f}(z_0))
		=
		\sigma(f_{n+1} \circ\dots\circ f_2(f_1(z)), f_{n+1} \circ\dots\circ f_2(f_1(z_0)))
		<\varepsilon;
	\]
	this is true for all~$\varepsilon>0$.
	In other words, the family~$\{F^{(n+1)}_\seq{f} \such n\in\mbN\}$ is spherically equicontinuous at~$z_0$.
	Consequently, the family~$\mcF_\seq{f} = \{F^{(n+1)}_\seq{f} \such n\in\mbN\} \cup \{f_1\}$	is also spherically equicontinuous at~$z_0$	and, by repeating the same argument with any other point of~$V$ in place of~$z_0$, at every point of~$V$.
	\cref{thm:Montel theorem spherical} implies, then, that~$\fat(\seq{f}) \supseteq V \ni z_0$, and in turn~$f_1^{-1}\fat(\shift \seq{f}) \subseteq \fat(\seq{f})$.
	
	Last, we prove the reverse inclusion~$\fat(\seq{f}) \subseteq f_1^{-1}\fat(\shift \seq{f})$.
	Let~$z_0\in \fat(\seq{f})$.
	Then, there exists a neighbourhood~$U$ of~$z_0$ where the family~$\mcF_\seq{f}$ is normal.
	Set~$V := f_1(U)$ and~$w_0 := f_1(z_0)$, and note that~$V$ is an open neighbourhood of~$w_0$ as~$f_1$ is an open mapping.
	Next, fix an~$\varepsilon>0$.
	By the normality of~$\mcF_\seq{f}$ on~$U$, there exists~$\delta_2>0$ such that, for all non-negative integers~$n$, for all~$z \in U$ with~$\sigma(z,z_0)<\delta_2$ it holds that
	\begin{equation}
	\label{eq:deterministic:normality at z around z_0}
		\sigma(F^{(n+1)}_\seq{f}(z),F^{(n+1)}_\seq{f}(z_0))
		=
		\sigma(
			f_{n+1} \circ\dots\circ f_1(z),
			f_{n+1} \circ\dots\circ f_1(z_0)
		)
		<\varepsilon.
	\end{equation}
	
	Now, the set~$f_1(U \cap D_\sigma(z_0,\delta_2))$ is also open and contains~$w_0$,
	hence one can find~$\delta'>0$ so that~$D_\sigma(w_0,\delta') \subseteq f_1(U \cap D_\sigma(z_0,\delta_2))$.
	With this in mind, for any~$w \in V$ with~$\sigma(w,w_0)<\delta'$ there exists a~$z\in U$ satisfying~$f_1(z)=w$ and~$\sigma(z,z_0)<\delta_2$,
	and in turn by \eqref{eq:deterministic:normality at z around z_0}
	\begin{align*}
		\sigma(F^{(n)}_{\shift \seq{f}}(w),F^{(n)}_{\shift \seq{f}}(w_0))
	&	=
		\sigma(
			f_{n+1} \circ\dots\circ f_2(w),
			f_{n+1} \circ\dots\circ f_2(w_0)
		)
	\\
	&	=
		\sigma(
			f_{n+1} \circ\dots\circ f_2(f_1(z)),
			f_{n+1} \circ\dots\circ f_2(f_1(z_0))
		)
	\\
	&	=
		\sigma(F^{(n+1)}_\seq{f}(z),F^{(n+1)}_\seq{f}(z_0))
		<\varepsilon
	\end{align*}
	for all~$n\in\mbN$.
	But then, the family~$\mcF_{\shift \seq{f}}=\{F^{(n)}_{\shift \seq{f}} \such n\in\mbN\}$ is spherically equicontinuous at~$w_0 = f_1(z_0)$, and therefore also at~$f_1(z)$ for any~$z \in U$ following the same reasoning.
	Hence,~$\fat(\shift \seq{f}) \supseteq f_1(U) \ni f_1(z_0)$,
	which implies~$z_0 \in f_1^{-1}(\fat(\shift \seq{f}))$,
	and in conclusion~$\fat(\seq{f}) \subseteq f_1^{-1}(\fat(\shift \seq{f}))$.
	
	Combining the two inclusions, we conclude that,~$\fat(\seq{f}) = f_1^{-1}(\fat(\shift \seq{f}))$.
\end{proof}

\begin{cor}
\label{cor:replacement_lemma}
	Let~$\seq{f}=(f_k)_{k\ge1}$ be a sequence of non-constant rational maps on~$\CC$ and~$g$ a non-constant rational map.
	Fix an~$n\in\mbN$,
	and define
	\[
		\seq{f}^{f_n \mapsto g} := (f_1,\dots,f_{n-1},g,f_{n+1},\dots).
	\]
	Then,
	\[
		\jul(\seq{f}^{f_n \mapsto g})
		= \big(F^{(n-1)}_\seq{f}\big)^{-1}
		g^{-1} \jul(\shift^n\seq{f})
		= \big(F^{(n-1)}_\seq{f}\big)^{-1} g^{-1} F^{(n)}_\seq{f}(\jul(\seq{f})).
	\]
	The same identity holds with~$\jul$ replaced by~$\fat$.
\end{cor}

It is well-known (see, for example, \cite[Theorem~4.2.1]{Beardon1991} or \cite[Lemma~4.8]{Milnor2006}) in the autonomous case that for a rational map with degree at least~$2$ the resultant Julia set is non-empty.
Below, we present the corresponding result for the Julia set of a sequence of holomorphic mappings.

\begin{prop}
\label{prop:Julia is non empty}
	Let~$\seq{f}=(f_1,f_2,\dots)\in\mcH(\CC)^\mbN$ be a sequence of non-constant holomorphic mappings.
	If~$\seq{f}$ contains infinitely many mappings of degree at least~$2$,
	then the Julia set~$\jul(\seq{f})$ is non-empty.
\end{prop}

\begin{proof}
	Assume that~$\jul(\seq{f})=\emptyset$.
	Then, the family~$\mcF_\seq{f} = \{F^{(n)}_\seq{f} \such n\in\mbN\}$ is normal on~$\CC$,
	and therefore the sequence~$(F^{(n)}_\seq{f})_{n\in\mbN}$ has a subsequence~$(F^{(n_k)}_\seq{f})_{k\in\mbN}$,
	which converges spherically uniformly on the compact subsets of~$\CC$, and in particular on~$\CC$, to some holomorphic mapping~$G\in\mcH(\CC)$.
	Per \cite[Theorem 2.8.2]{Beardon1991},~$G$ is in fact a rational map of some (fixed) non-negative degree~$d \in \mbN$, and the functions~$F^{(n_k)}_\seq{f}$ eventually also have degree exactly~$d$.
	However,~$F^{(n_k)}_\seq{f} = f_{n_k} \circ\dots\circ f_1$ and therefore~$\deg F^{(n_k)}_\seq{f} = \deg f_{n_k} \dots \deg f_1$.
	If there are infinitely many maps~$f_i$ of degree at least~$2$,
	the above quantity tends to infinity as~$n_k\to+\infty$.
	This is a contradiction,
	and consequently~$\jul(\seq{f})$ cannot be empty.
\end{proof}


\section{Julia sets of monomial sequences}
\label{sec:Monomials}


\subsection{Precise characterisation of the Julia set}
\label{sec:classification}


In this section, we describe the non-autonomous dynamics of monomial sequences.
In this setting, the dynamics are purely radial, and the Julia set is determined by the accumulation points of the \emph{real} sequence~$s_n = \frac{1}{D_n} \log|A_n|$.

We begin with the behaviour at~$0$ and~$\infty$ before stating our main \cref{thm:monomial_julia_accumulation}.

\begin{lem}
\label{lem:monomial_zero_infty_fatou}
C	Consider a sequence of monomials~$\seq{p}=(p_n)_{n\ge1}$ with~$p_n(z)=a_n z^{d_n}$,~$a_n\in\mbC\setminus\{0\}$ and~$d_n\ge1$.
	Set
	\[
		D_n := \prod\nolimits_{i=1}^n d_i,
		\quad A_n := \prod\nolimits_{i=1}^n a_i^{D_n/D_i}
		\inline{and}
		s_n := \frac{1}{D_n} \log|A_n| = \sum\nolimits_{i=1}^n \frac{1}{D_i} \log|a_i|.
	\]
	Then,
	\[
		0\in\fat(\seq{p}) \iff \limsup\nolimits_{n\ge1}s_n < +\infty,
	\inline{and}
		\infty\in\fat(\seq{p}) \iff \liminf\nolimits_{n\ge1}s_n > -\infty.
	\]
\end{lem}

\begin{proof}
	We first treat the point at~$0$.
	Assume~$\limsup_{n\geq1} s_n < +\infty$ and notice it is equivalent to~$\sup_{n\ge1}s_n < +\infty$, since~$s_n \in \mbR$ for all~$n\in\mbN$.
	Let us set
	\[
		M := \sup\{s_n \such n\geq1\} \inline{and} r := e^{-(M+1)}.
	\]
	Then, for every~$n\ge1$ and every~$|z|\le r$,
	\[
		|F^{(n)}_{\seq{p}}(z)|
		=|A_n|\,|z|^{D_n}
		= e^{D_n s_n} e^{D_n\log|z|}
		\le e^{D_n(M+\log r)}
		= e^{-D_n}
		\le e^{-1},
	\]
	and the family~$\{F^{(n)}_{\seq{p}} \such n\geq1\}~$ is (locally) bounded on~$D(0,r)$.
	Therefore, it is a normal family on this set thanks to \cref{thm:Montel for holomorphic};
	in particular~$0\in\fat(\seq{p})$.
	(If additionally~$d_n\geq 2$ for infinitely many~$n$,
	then~$D_n\to+\infty$ and~$F^{(n)}_{\seq{p}}\to 0$ locally uniformly on~$D(0,r)$.)
	
	Conversely, assume~$0\in\fat(\seq{p})$.
	Then,~$\{F^{(n)}_{\seq{p}}\}_{n\ge1}$ is normal in a neighbourhood of~$0$ hence equicontinuous there by \cref{thm:Montel theorem spherical}.
	Take an~$\varepsilon\in(0,\sqrt2)$.
	Since~$F^{(n)}_{\seq{p}}(0)=0$ for all~$n$,
	by the equicontinuity there exists~$\delta>0$ such that
	\[
		\sigma(z,0)<\delta
		\implies
		\sigma\big(F^{(n)}_{\seq{p}}(z),0\big) < \varepsilon
		\afterline{for all} n\ge1.
	\]
	Fix any such~$z_0 \in D_\sigma(0,\delta) \setminus \{0\}$.
	Then, for all~$n\geq1$, \eqref{eq:chordal less than spherical} gives
	\[
		\frac{2|F^{(n)}_{\seq{p}}(z_0)|}{\sqrt{1+|F^{(n)}_{\seq{p}}(z_0)|^2}} < \varepsilon
		\quad\iff\quad
		|A_n||z_0|^{D_n} = |F^{(n)}_{\seq{p}}(z_0)| < \frac{\varepsilon}{\sqrt{4-\varepsilon^2}} =: \alpha < 1
	\]
	so that
	\[
		s_n
		= \frac{1}{D_n} \log|A_n|
		< \frac{1}{D_n} \log\alpha - \log|z_0|
		< -\log |z_0|.
	\]
	Hence,~$\sup_{n\ge1}s_n$, thus also~$\limsup_{n\geq1} s_n$, is finite.
	
	The statement for~$\infty$ can be treated similarly,
	but we present here a different argument utilising the previous result.
	Consider the M{\"o}bius transformation~$\mu(z) = z^{-1}$ and a new sequence of monomials~$\tilde{\seq{p}} = (\tilde{p}_1,\tilde{p}_2,\dots)$ defined by
	\[
		\tilde p_n(w)=\mu^{-1} \circ p_n \circ \mu (w) = a_n^{-1} w^{d_n}.
	\]
	The sequence~$\tilde{s}_n$ for~$\tilde{\seq{p}}$ corresponding to~$s_n$ is, then, given by
	\[\tilde{s}_n = \frac{1}{D_n} \log |\prod\nolimits_{i=1}^n (a_i^{-1})^{D_n/D_i}| = -s_n,\]
	and from the first statement, we get that
	\[
		0\in\fat(\tilde{\seq{p}})
		\quad\iff\quad
		\limsup\nolimits_{n\ge1}(-s_n) < +\infty
		\quad\iff\quad
		\liminf\nolimits_{n\ge1}s_n > -\infty.
	\]
	However, notice that~$\fat(\tilde{\seq{p}}) = \mu^{-1} \fat(\seq{p})$,
	thus~$\infty\in\fat(\seq{p})$ if, and only if,~$\liminf_{n\geq1} s_n > -\infty$.
\end{proof}

Notice that~$\limsup$ and~$\liminf$ can be readily replaced by~$\sup$ and~$\inf$.

Observe that \cref{lem:monomial_zero_infty_fatou} does not require the degrees of the monomials to be at least 2 for infinitely many of them, as opposed to \cref{prop:Julia is non empty}.
Yet, the Julia set is guaranteed to be non-empty whenever~$\limsup_{n\ge1} s_n = +\infty$ or~$\liminf_{n\ge1} s_n = -\infty$.

\monomialjulia

\begin{proof}
	We first prove the inclusion~$\bigcup_{s\in S}\{z\in\CC \such |z|=e^{-s}\} \subseteq \jul(\seq{p})$.
	If~$S\cap\mbR$ is non-empty,
	let~$s\in S\cap\mbR$.
	Choose a subsequence of~$(s_n)$ such that~$s_{n_k}\to s$ as~$n_k\to+\infty$,
	and fix a~$\theta\in[0,2\pi)$.
	We set
	\[
		g_k := F^{(n_k)}_{\seq{p}},
	\quad
		z := e^{-s} e^{i\theta},
	\inline{and}
		z_k := e^{-s_{n_k}}e^{i\theta}.
	\]
	Immediately, we have that~$z_k\to z$ and
	\[
		|g_k(z_k)|
		= |A_{n_k}| \, |z_k|^{D_{n_k}}
		= e^{D_{n_k}s_{n_k}} \, (e^{-s_{n_k}})^{D_{n_k}}
		= 1,
	\]
	and in turn we have
	\[
		|g_k'(z_k)|
		= D_{n_k}|A_{n_k}|\,|z_k|^{D_{n_k}-1}
		= D_{n_k} / |z_k|.
	\]
	Since~$D_n\to+\infty$, as~$z_k \to z$ we have that~$|g_k'(z_k)|\to+\infty$,
	and therefore,
	\[
		g_k^\#(z_k)
		= \frac{2|g_k'(z_k)|}{1+|g_k(z_k)|^2}
		= |g_k'(z_k)|
		\to +\infty.
	\]
	By \cref{thm:Marty}, the family is not normal at~$z$.
	Since the choice of~$\theta$ was arbitrary, every point of the circle~$\{|z|=e^{-s}\}$ lies in the Julia set~$\jul(\seq{p})$.
	If~$+\infty \in S$, then~$\limsup_{n\to+\infty} s_n = +\infty$,
	and \cref{lem:monomial_zero_infty_fatou} implies that~$0\in\jul(\seq{p})$.
	Likewise, if~$-\infty \in S$, then~$\liminf_{n\to+\infty} s_n = -\infty$,
	and again \cref{lem:monomial_zero_infty_fatou} implies~$\infty\in\jul(\seq{p})$.
	
	We now proceed to prove the reverse inclusion~$\jul(\seq{p}) \subseteq \bigcup_{s\in S}\{z\in\CC \such |z|=e^{-s}\}$.
	If~$\jul(\seq{p})\setminus\{0,\infty\}$ is non-empty,
	let us consider a point~$z_0\in\jul(\seq{p})\setminus\{0,\infty\}$ and set~$t := -\log|z_0|$.
	We claim that~$t\in S$.
	Assume not.
	Then, for some sufficiently small~$\eta>0$ there exists an~$N\in\mbN$ such that~$|s_n-t|\ge 2\eta$ for all~$n \geq N$.
	We set
	\[
		r := |z_0|e^{-\eta},
		\qquad
		R := |z_0|e^\eta,
		\inline{and}
		A := \{z\in\mbC \such r<|z|<R\},
	\]
	from which we have
	\[
		-\log|z| \in (t-\eta,t+\eta) \afterline{for all} z \in A.
	\]
	Note that~$z_0 \in A$ as well.
	We subsequently get the following two estimates:
	For any~$s_n$ with~$s_n\ge t+2\eta$ and for every~$z\in A$,
	\[
		s_n+\log|z|
		= s_n-(-\log|z|)
		\ge (t+2\eta)-(t+\eta)=\eta
	\implies
		|F^{(n)}_{\seq{p}}(z)|
		= e^{D_n(s_n+\log|z|)}
		\ge e^{\eta D_n};
	\]
	similarly, for any~$s_n$ with~$s_n\le t-2\eta$ and for every~$z\in A$,
	\[
		s_n+\log|z|
		\le (t-2\eta)-(t-\eta)
		= -\eta
	\implies
		|F^{(n)}_{\seq{p}}(z)|
		\le e^{-\eta D_n}.
	\]
	However, given any subsequence~$s_{n_k}$ of~$s_n$,
	there must be infinitely many indices~$n_k \geq N$ for which~$s_{n_k}$ satisfies at least one of the inequalities~$s_{n_k} \geq t+2\eta$ and~$s_{n_k} \leq t-2\eta$.
	Therefore, there is always a further subsequence ---again denoted by~$s_{n_k}$--- such that
	\[
		\text{either}\ F^{(n_k)}_{\seq{p}}\to\infty \inmath{uniformly on} A
		\quad\text{or}\ F^{(n_k)}_{\seq{p}}\to0 \inmath{uniformly on} A,
	\]
	since~$D_{n_k}\to+\infty$ as~$n_k\to+\infty$.
	Hence, the family~$\{F^{(n)}_{\seq{p}}\}_{n\ge1}$ is normal on~$A$.
	But this contradicts the fact that~$z_0\in\jul(\seq{p})$,
	and therefore~$t\in S$ and eventually
	\[
		\jul(\seq{p})\setminus\{0,\infty\}
		\subseteq
		\bigcup\nolimits_{s\in S}\{z\in\CC \such |z|=e^{-s}\}.
	\]
	Finally, if~$0\in\jul(\seq{p})$,
	\cref{lem:monomial_zero_infty_fatou} gives us that~$\limsup_{n\ge1}s_n=+\infty$;
	thus~$+\infty\in S$ and~$0\in\bigcup_{s\in S}\{z\in\CC \such |z|=e^{-s}\}$.
	Likewise, if~$\infty\in\jul(\seq{p})$,~$\liminf_{n\ge1}s_n=-\infty\in S$, and~$\infty\in\bigcup_{s\in S}\{z\in\CC \such |z|=e^{-s}\}$.
\end{proof}

\begin{cor}
\label{cor:monomial_isolated_points}
	Under the assumptions of \cref{thm:monomial_julia_accumulation},~$\jul(\seq{p})\setminus\{0,\infty\}$ is a (possibly empty) union of circles centred at~$0$,
	and every point of~$\jul(\seq{p})\setminus\{0,\infty\}$ is non-isolated.
	Hence, the only possible isolated points of~$\jul(\seq{p})$ are~$0$ and/or~$\infty$.
	Moreover,
	\begin{enumerate}
		\item $0$ is isolated in~$\jul(\seq{p})$ if, and only if,~$+\infty \in S$ and~$+\infty$ is an isolated point of~$S$;
		\item $\infty$ is isolated in~$\jul(\seq{p})$ if, and only if,~$-\infty\in S$ and~$-\infty$ is an isolated point of~$S$.
	\end{enumerate}
\end{cor}

\begin{proof}
	If~$z\in\jul(\seq{p})\setminus\{0,\infty\}$,
	then by \cref{thm:monomial_julia_accumulation} the entire circle~$\{w\in\mbC \such|w|=|z|\}$ is contained in~$\jul(\seq{p})$,
	and so~$z$ is not isolated.
	Further,~$0\in\jul(\seq{p})$ is isolated exactly when there is no sequence of circles~$C(0,e^{-s_j})$ contained in~$\jul(\seq{p})$ with~$s_j \to +\infty$,
	i.e.\ when~$+\infty$ is isolated in~$S$.
	The statement for~$\infty$ is analogous.
\end{proof}

\begin{prop}[The eventually linear case]
\label{prop:monomial_eventually_linear}
	Let~$\seq{p}=(p_n)_{n\ge1}$ be a sequence of monomials with~$p_n(z)=a_n z^{d_n}$,~$a_n\in\mbC\setminus\{0\}$ and~$d_n\ge1$,
	and assume that~$d_n=1$ for all~$n\ge n_0$ for some~$n_0\in\mbN$.
	Then,~$\jul(\seq{p})\subseteq\{0,\infty\}$.
	In particular,
	\[
		0\in\jul(\seq{p}) \iff \limsup\nolimits_{n\geq1}|A_n| = +\infty
		\inline{and}
		\infty\in\jul(\seq{p}) \iff \liminf\nolimits_{n\geq1}|A_n| = 0.
	\]
\end{prop}

We note that in this case the Julia set may very well be empty.

\begin{proof}
	By the assumptions,~$D_n = D_{n_0}$ for all~$n \geq n_0$ so that~$F^{(n)}_{\seq{p}}(z)=A_n z^{D_{n_0}}$ for all~$n\ge n_0$.
		
	We first show that~$\jul(\seq{p})\subseteq\{0,\infty\}$.
	Let~$K\subseteq\CC\setminus\{0,\infty\}$ be compact.
	Then, there exist constants~$0 < m \leq M < +\infty$ such that~$m \leq |z| \leq M$ for all~$z\in K$.
	Consider an arbitrary subsequence~$\big(F^{(n_k)}_{\seq{p}}\big)$ of~$\big(F_\seq{p}^{(n)})$;
	we may assume that~$n_k \geq n_0$ for all~$k\geq1$.
	Since~$\CC$ is compact, the sequence~$(A_{n_k})_{k \geq 1}$ has a convergent subsequence in~$\CC$
	(again denoted by~$A_{n_k}$) such that~$A_{n_k}\to A$ for some~$A\in\CC$.
	If~$A\in\mbC$, then
	\[
		\sup\nolimits_{z\in K}\big|F^{(n_k)}_{\seq{p}}(z)-Az^{D_{n_0}}\big|
		\leq |A_{n_k}-A|\,M^{D_{n_0}}
		\to 0
		\afterline{as} n_k\to+\infty;
	\]
	if~$A=\infty$, then
	\[
		|F^{(n_k)}_{\seq{p}}(z)|
		= |A_{n_k}|\,|z|^{D_{n_0}}
		\geq |A_{n_k}|\,m^{D_{n_0}}
		\to +\infty
		\afterline{as} n_k\to+\infty
	\]
	uniformly on~$K$.
	Thus, the family is normal on~$\CC\setminus\{0,\infty\}$,
	i.e.~$\jul(\seq{p})\subseteq\{0,\infty\}$.
	
	Now, notice that
	\[s_n = \frac{1}{D_{n_0}} \log|A_n| \afterline{for all} n \geq n_0.\]
	The rest of the claims follow directly from \cref{lem:monomial_zero_infty_fatou}.
\end{proof}


\subsection{Orbital description of a Julia set of a monomial sequence}
\label{sec:orbital description}


In the autonomous setting, the Julia set of a holomorphic mapping equals the closure of the (full) backward orbit of any of its points.
In the non-autonomous setting, we can in fact obtain a similar dynamical description for the Julia set of a sequence of monomials.

\begin{thm}
\label{thm:derived set is julia}
	Consider a sequence of monomials~$\seq{p}=(p_n)_{n\ge1}$ with~$p_n(z)=a_n z^{d_n}$,~$a_n\in\mbC\setminus\{0\}$ and~$d_n\ge1$,
	and assume that~$d_n\ge2$ for infinitely many~$n$.
	Set~$D_n := \prod_{i=1}^n d_i$,
	and let~$\seq{w} = (w_1, w_2, \dots)$ be a sequence of points in~$\CC\setminus \{0,\infty\}$ such that 
	\[
		\lim_{n\to+\infty} \frac{1}{D_n} \log{|w_n|} = 0.
	\]
	Consider the set
	\[
		E(\seq{w})
		:= \bigcup\nolimits_{n=1}^{+\infty} \big(F^{(n)}_{\seq{p}}\big)^{-1}\{w_n\},
	\]
	and let~$E'(\seq{w})$ be the derived set of~$E(\seq{w})$.
	It holds that 
	\[
		\jul(\seq{p}) = E'(\seq{w}).
	\]
	
	In particular, for a point~$w\in\CC\setminus\{0,\infty\}$ it holds that~$\jul(\seq{p}) = E'(w)$.
\end{thm} 

\begin{proof}
	Using \cref{thm:monomial_julia_accumulation}, it is sufficient to show that~$E'(\seq{w}) = \bigcup_{s\in S} \{z\in\CC \such |z| = e^{-s}\}$.
	
	First note that, for any fixed~$n\in\mbN$,~$\big(F^{(n)}_{\seq{p}}\big)^{-1}\{w_n\}$ is a collection of~$D_n$ many equidistant points on the circle~$\{|z| = \rho_n\}$ where 
	\[
		\rho_n
		= |w_n / A_n|^{1/D_n} 
		= e^{-s_n + \frac{1}{D_n} \log|w_n|};
	\]
	specifically, if~$\theta_n\in[0,2\pi)$ is the principal argument of~$\frac{w_n}{A_n}$,
	\[
		\big(F^{(n)}_{\seq{p}}\big)^{-1}\{w_n\}
		=
		\{ z_{n,k} := \rho_n e^{i(\theta_n + 2\pi k)/D_n} \such 0 \le k \le D_n - 1 \}.
	\]
	
	Now, let~$s\in S$.
	Pick a subsequence~$(s_{n_k})$ of~$(s_n)$ with~$s_{n_k}\to s$ as~$k\to+\infty$,
	and consider the (sub)sequence~$(\rho_{n_k})_{k\ge 1}$.
	Then,~$\rho_{n_k} \to e^{-s}$ as~$k\to+\infty$, since~$\lim_{n\to+\infty} \frac{1}{D_n} \log|w_n| = 0$.
	If~$s \in \mbR$,
	take any point on the circle~$\{|z| = e^{-s}\} \subseteq \jul(\seq{p})$, say~$z_0 = e^{-s} e^{i\theta}$ for some~$\theta\in[0,2\pi)$.
	We claim that~$z_0 \in E'(\seq{w})$.
	Without loss of generality, we can assume that~$D_{n_k}\geq2$ for all~$k\geq1$.
	Then, for each~$k\geq1$ select~$j_k \in \{0, \dots , D_{n_k}-1\}$ so that
	\[
		\Big|\theta - \frac{\theta_{n_k} + 2\pi j_k}{D_{n_k}}\Big| \le \frac{2\pi}{D_{n_k}}.
	\]
	If~$z_{n_k,j_k} \neq z_0$ for infinitely many~$k$,
	we can find a subsequence,~$(y_k)$ of~$z_{n_k,j_k}$ so that~$y_k \neq z_0$ for all~$k$.
	If~$z_{n_k,j_k} = z_0$ for eventually for~$k$,
	we instead consider the sequence~$y_k = z_{n_k,j_k+1}$ for which~$y_k \neq z_0$ eventually for~$k$.
	In either case, observe that as~$k\to+\infty$ we have~$|y_k| = \rho_{n_k}\to e^{-s}$ and~$\Arg y_k \to \theta$,
	that is,~$y_k \to z_0$.
	Additionally,~$y_k \in E(\seq{w})\setminus\{z_0\}$ eventually for~$k$,
	and therefore~$z_0 \in E'(\seq{w})$.
	If~$s = +\infty$ (that is,~$0\in\jul(\seq{p})$), then~$\rho_{n_k} \to 0$,
	and working as above for the sequence~$(y_k)_{k\geq1}$ given by~$y_k = z_{n_k,0} \neq 0$, we get that~$0\in E'(\seq{w})$.
	Similarly, if~$s = -\infty$ (that is,~$\infty\in\jul(\seq{p})$),~$\rho_{n_k} \to +\infty$ and we obtain~$\infty \in E'(\seq{w})$.
	This establishes that 
	\[
		\jul(\seq{p})
		= \bigcup\nolimits_{s\in S} \{z\in\CC \such |z| = e^{-s}\}
		\subseteq E'(\seq{w}).
	\]
	
	This additionally implies that~$E'(\seq{w})$ is non-empty thanks to \cref{prop:Julia is non empty}.
	
	For the converse, let~$z\in E'(\seq{w})$,
	and~$(z_k)_{k\ge 1}$ be a sequence in~$E(\seq{w})$ of distinct points so that~$z_k \to_\sigma z$.
	Since~$z_k \in E(\seq{w})$, we get~$z_k \in \big(F^{(n_k)}_{\seq{p}}\big)^{-1}\{w_{n_k}\}$ for some~$n_k\in\mbN$.
	However, each~$\big(F^{(n_k)}_{\seq{p}}\big)^{-1}\{w_{n_k}\}$ is finite and the~$z_k$'s are distinct,
	which means that~$n_k$ can take each value in~$\mbN$ at most finitely many times;
	that is,~$n_k\to+\infty$ as~$k\to+\infty$.
	At the same time,
	\[
		|z_k| = \rho_{n_k} = e^{-s_{n_k} + \frac{1}{D_{n_k}} \log |w_{n_k}|}
		\inline{and}
		|z_k| \xlongrightarrow{k\to+\infty} |z| \inmath{in} \overline{\mbR},
	\]
	Because~$\lim_{n\to+\infty} \frac{1}{D_n}\log{|w_n|} = 0$ by assumption,
	the sequence~$(s_{n_k})$ necessarily converges to~$-\log{|z|} \in \overline{\mbR}$ as~$k\to+\infty$.
	This implies that~$-\log{|z|} \in S$,
	and therefore~$z \in \{|w| = e^{-(-\log|z|)}\} \subseteq \jul(\seq{p})$ concluding the desired set inclusion.
\end{proof} 

In the case when the elements of~$\seq{w}$ lie inside the Julia sets, we get the following result.

\begin{cor}
	Consider a sequence of monomials~$\seq{p}=(p_n)_{n\ge1}$ with~$p_n(z) = a_n z^{d_n}$,~$a_n\in\mbC\setminus\{0\}$ and~$d_n\ge1$,
	and assume that~$d_n\ge2$ for infinitely many~$n$.
	Set~$D_n := \prod_{i=1}^n d_i$.
	Assume there is a sequence of points~$\seq{w} = (w_1, w_2, \dots)$ with~$w_n \in \jul(\shift^n \seq{p})\setminus\{0,\infty\}$ for each~$n\in\mbN$ and so that
	\[
		\lim_{n\to+\infty} \frac{1}{D_n} \log{|w_n|} = 0.
	\]
	Then, the Julia set~$\jul(\seq{p})$ is perfect.
\end{cor}

\begin{proof}
	Consider the set~$E(\seq{w}) := \bigcup_{n=1}^{+\infty} \big(F^{(n)}_{\seq{p}}\big)^{-1}\{w_n\}$.
	Let~$w_n \in \jul(\tau^n\seq{p})$ for each~$n\in\mbN$ and denote~$\seq{w} = (w_1, w_2,\dots)$.
	From \cref{thm:invariance deterministic}, we have that~$(F^{(n)}_{\seq{p}})^{-1} \jul(\shift^n {\seq{p}}) = \jul(\seq{p})$
	therefore~$\big(F^{(n)}_{\seq{p}}\big)^{-1}\{w_n\}\subseteq\jul(\seq{p})$.
	This gives us that~$E(\seq{w}) \subseteq\jul(\seq{p})$,
	and therefore~$E'(\seq{w}) \subseteq \jul'(\seq{p})$,
	where~$E'(\seq{w})$ and~$\jul'(\seq{p})$ are the derived sets of~$E(\seq{w})$ and~$\jul(\seq{p})$ respectively.
	On the other hand, from \cref{thm:derived set is julia} we get that~$E'(\seq{w}) = \jul(\seq{p})$, where~$E'(\seq{w})$ is the derived set of~$E(\seq{w})$,
	and consequently,~$\jul(\seq{p}) \subseteq \jul'(\seq{p})$.
	As~$\jul(\seq{p})$ is closed, this implies~$\jul(\seq{p}) = \jul'(\seq{p})$,
	and~$\jul(\seq{p})$ is perfect.
\end{proof}


\subsection{Julia sets as Kuratowski limits superior}
\label{sec:Kuratowski}


Consider the monomial~$p(z) = az^d$ where~$a\in\mbC\setminus\{0\}$ and~$d\geq2$.
It is a standard computation to calculate its autonomous Julia set,~$\jul^\mathrm{aut}(p)$,
which is a circle centred at~$0$ with radius~$|a|^{-\frac{1}{d-1}}$:
\begin{equation}
\label{eq:autonomous Julia of monomial}
	\jul^\mathrm{aut}(p) = \{z\in\mbC \such |z| = |a|^{-\frac{1}{d-1}}\}.
\end{equation}

It turns out that for monomial sequences one can retrieve the non-autonomous Julia set as the Kuratowski limit superior of the autonomous Julia sets produced by each non-autonomous iterate,
even though the autonomous and non-autonomous Julia sets need not have any common points.
We make this precise in the theorem below.

But first, we recall the Kuratowski (or Painlev{\'e}-Kuratowski) convergence:
Given a sequence of subsets,~$(E_n)_{n\geq1}$, of a metric space~$(X,d)$, their \emph{(Kuratowski) limit superior} is defined by
\[
	\limsup\nolimits_{n\to+\infty} E_n = 
		\big\{
			x \in X \such \liminf\nolimits_{n\to+\infty} \dist_d(x,E_n) = 0
		\big\}.
\]
The interested reader can look into \cite[\S4]{Rockafellar1998} for these definitions and more details.
One additional property that we will use is the following:
For any sequence of sets~$(E_n)_{n\geq1}$ in~$(X,d)$ it holds that
\begin{equation}
\label{eq:Kuratowski limsup is intersection of unions}
	\limsup\nolimits_{n\to+\infty}E_n
	= \bigcap\nolimits_{N\ge1}\overline{\bigcup\nolimits_{n\ge N}E_n}.
\end{equation}

For this subsection, all set-theoretic limits are understood as Kuratowski limits in~$(\CC,\sigma)$.

\begin{thm}
\label{thm:Kuratowski}
	Let~$\seq{p}=(p_n)_{n\ge1}$ be a sequence of monomials~$p_n(z)=a_n z^{d_n}$ with~$a_n\in\mbC\setminus\{0\}$ and~$d_n\ge1$ for all~$n\in\mbN$,
	and assume that~$d_n\ge2$ for infinitely many~$n$.
	Denote by~$F^{(n)}_{\seq{p}}$ the~$n$-fold iterative composition~$F^{(n)}_{\seq{p}} = p_n\circ\dots\circ p_1$,
	and set
	\[
		\jul^{\mathrm{aut}}\big(F^{(n)}_\seq{p}\big) := \inmath{the autonomous Julia set of the function} F^{(n)}_{\seq{p}}.
	\]
	Then, we have that 
	\[
		\jul(\seq{p}) 
		= \limsup\nolimits_{n\to+\infty} \jul^{\mathrm{aut}}\big(F^{(n)}_\seq{p}\big)
		= \bigcap\nolimits_{N\ge 1}\overline{\bigcup\nolimits_{n\ge N} \jul^{\mathrm{aut}}\big(F^{(n)}_\seq{p}\big)},
	\]
	where the Kuratowski limit superior and the closure are taken in the metric space~$(\CC,\sigma)$.
\end{thm}

\begin{proof}
	Let~$N_0\in\mbN$ be such that~$D_N\geq2$ for all~$N \geq N_0$.
	Because~$\cup_{n \ge N} \jul^{\mathrm{aut}}\big(F^{(n)}_\seq{p}\big)$ is a decreasing sequence of sets,
	the following equality
	\[
		\bigcap\nolimits_{N\ge1} \overline{\bigcup\nolimits_{n \ge N} \jul^{\mathrm{aut}}\big(F^{(n)}_\seq{p}\big)}
		= \bigcap\nolimits_{N \ge N_0} \overline{\bigcup\nolimits_{n \ge N} \jul^{\mathrm{aut}}\big(F^{(n)}_\seq{p}\big)}.
	\]
	holds true.
	Considering \cref{thm:monomial_julia_accumulation} and \eqref{eq:Kuratowski limsup is intersection of unions}, it remains to show that
	\[
		\bigcap\nolimits_{N \ge N_0} \overline{\bigcup\nolimits_{n \ge N} \jul^{\mathrm{aut}}\big(F^{(n)}_\seq{p}\big)}
		= \bigcup\nolimits_{s\in S} \{z\in\CC \such |z| = e^{-s}\}.
	\]
	
	Choose~$N_0 \in \mbN$ such that~$D_n \ge 2$ for each~$n\ge N_0$.
	Now, let us define
	\[
		t_n := \frac{1}{D_n-1} \log |A_n| = \frac{D_n}{D_n-1} s_n \afterline{for all} n\ge N_0.
	\]
	Because~$\frac{D_n}{D_n-1}\to1$ as~$n\to+\infty$,
	for every converging (in~$\overline{\mbR}$) subsequence of~$(s_n)$ the corresponding subsequence of~$(t_n)$ also converges (in~$\overline{\mbR}$) to the same limit,
	and vice versa.
	Therefore, the set of accumulation points~$T\subseteq\overline{\mbR}$ of~$(t_n)_{n\ge N_0}$ is the same as the set of accumulation points~$S$ of~$(s_n)$,
	i.e.~$T=S$.
	Hence, it is sufficient, and we proceed to show that
	\[
		\bigcap\nolimits_{N \ge N_0} \overline{\bigcup\nolimits_{n\ge N} \jul^{\mathrm{aut}}\big(F^{(n)}_\seq{p}\big)}
		= \bigcup\nolimits_{t\in T} \{|z| = e^{-t}\}.
	\]
	
	Let~$z \in \bigcap_{N \ge N_0} \overline{\bigcup_{n\ge N} \jul^{\mathrm{aut}}\big(F^{(n)}_\seq{p}\big)}$,
	and set~$t := -\log|z| \in \overline{\mbR}$.
	We may then recursively choose a strictly increasing sequence of indices~$(n_k)_{k\ge 1}$, and points~$z_{n_k}\in \jul_{n_k}^{\mathrm{aut}}(\seq{p})$ (with~$n_k \geq N_0$) such that~$z_{n_k} \to_\sigma z$ as~$k\to+\infty$.
	By the definition of~$t_n$ and \eqref{eq:autonomous Julia of monomial}, we see that~$|z_{n_k}| = |A_{n_k}|^{-\frac{1}{D_{n_k}-1}} = e^{-t_{n_k}}$,
	and therefore,~$(t_{n_k})$ must converge to~$-\log|z| = t$ in~$\overline{\mbR}$.
	Thus,~$t \in T$ and~$|z|=e^{-t} \in \cup_{t \in T} \{|z| = e^{-t}\}$.
	
	Conversely, let~$t\in T$, and consider a point~$z\in\CC$ with~$|z|=e^{-t}$.
	There is a subsequence~$(t_{n_k})$ of~$(t_n)$ such that~$t_{n_k}\to t$.
	If~$0<|z|<+\infty$,
	then~$z=|z|e^{i\theta}$ for some~$\theta\in[0,2\pi)$,
	and notice that~$z_{n_k} := e^{-t_{n_k}} e^{i\theta}\in J^{\mathrm{aut}}_{n_k}(\seq{p})$;
	then~$z_{n_k}\to z$.
	If~$|z|=0$, then~$t = +\infty$,
	and take the sequence~$z_{n_k} := e^{-t_{n_k}}$; again~$z_{n_k}\to_\sigma z$.
	Similarly for when~$z = \infty$.
	Hence,~$z\in\bigcap_{N \ge N_0} \overline{\bigcup_{n\ge N} \jul^{\mathrm{aut}}\big(F^{(n)}_\seq{p}\big)}$
	and the theorem follows.
\end{proof}

Combining the ideas in \cref{thm:derived set is julia,thm:Kuratowski}, we get the following additional description.

\begin{thm}
\label{thm:Kuratowski_cand_derived}
	Consider a sequence of monomials~$\seq{p}=(p_n)_{n\ge1}$ with~$p_n(z)=a_n z^{d_n}$,~$a_n\in\mbC\setminus\{0\}$ and~$d_n\ge1$,
	and assume that~$d_n\ge2$ for infinitely many~$n$.
	Let~$\seq{w} = (w_1, w_2, \dots)$ be a sequence of points in~$\CC\setminus \{0,\infty\}$, and define
	\[
		E_n(\seq{w}) := \big(F_{\seq{p}}^{(n)}\big)^{-1}\{w_n\}
		\inline{and}
		E(\seq{w}) := \bigcup\nolimits_{n=1}^{+\infty}E_n(\seq{w}).
	\]
	Then, for the Kuratowski limit superior in~$(\CC,\sigma)$ it holds that
	\begin{equation}
	\label{eq:derived set is Kuratowski limsup}
		E'(\seq{w})
		= \limsup\nolimits_{n\to+\infty} E_n(\seq{w}), 
	\end{equation}
	where~$E'(\seq{w})$ is the derived set of~$E(\seq{w})$.
	If the sequence~$\seq{w}$ additionally satisfies
	\[
		\lim_{n\to+\infty} \frac{1}{D_n} \log{|w_n|} = 0,
	\]
	then~$\jul(\seq{p}) = \limsup_{n\to+\infty} \big(F_{\seq{p}}^{(n)}\big)^{-1}\{w_n\}$.
\end{thm}

We remark that the set theoretic identity \eqref{eq:derived set is Kuratowski limsup} need not hold in general for polynomials.

\begin{proof}
	The latter claim follows directly from \cref{thm:derived set is julia} once \eqref{eq:derived set is Kuratowski limsup} is established.
	Therefore, we proceed towards this goal.
	
	We first notice that each~$E_n(\seq{w})=\big(F_{\seq{p}}^{(n)}\big)^{-1}\{w_n\}$ is finite containing exactly~$D_n$ distinct and equidistant points in~$\CC\setminus\{0,\infty\}$.
	Let~$z \in E'(\seq{w})$.
	There exists a sequence of \emph{distinct} points~$(z_k)_{k\geq1}$ in~$E(\seq{w})$ such that~$z_k \to_\sigma z$ as~$k\to+\infty$,
	and therefore for each~$k\in\mbN$ there is some~$n_k\in\mbN$ such that~$z_k \in E_{n_k}(\seq{w})$.
	Since the~$E_n(\seq{w})$'s are finite sets and the~$z_k$'s distinct,
	the indices~$n_k$ can take each value in~$\mbN$ at most finitely many times;
	this implies that~$n_k\to+\infty$ as~$k\to+\infty$.
	Immediately we get that~$z \in \limsup_{n\to+\infty} E_n(\seq{w})$.
	
	Conversely, let~$z\in \limsup_{n\to+\infty}E_n(\seq{w})$.
	Then, by the definition there must exist a sequence~$(n_k)_{k\geq1}$ (with~$n_k\to+\infty$) and points~$z_{k}\in E_{n_k}(\seq{w})$ such that~$z_k \to_\sigma z$.
	Next, there are two possibilities:
	either~$z_k \neq z$ for infinitely many~$k$ or~$z_k = z$ eventually for~$k$.
	In the former case, we may pass to a subsequence~$(z_{k_j})_{j\geq1}$ such that~$z_{k_j}\neq z$ for all~$j$ and~$z_{k_j} \to_\sigma z$ as~$j\to+\infty$,
	hence~$z \in E'(\seq{w})$.
	In the latter case, we may assume that~$z_k=z$ for all~$k$.
	Since~$E_{n_k}(\seq{w})\subseteq\CC\setminus\{0,\infty\}$,
	necessarily also~$z\in\CC\setminus\{0,\infty\}$.
	Now, let us define
	\[
		\zeta_k
		:= e^{\frac{2\pi i}{D_{n_k}}} z_k = e^{\frac{2\pi i}{D_{n_k}}} z \afterline{for all} k\geq1.
	\]
	We may assume without loss of generality that for all~$k\geq1$ we have~$D_{n_k}\geq2$ so that~$\zeta_k \neq z$.
	Yet,~$\zeta_k \to z$,
	since~$D_n\to+\infty$ as~$n\to+\infty$;
	thus also~$\zeta_k \to_\sigma z$ as~$\zeta_k, z\neq \infty$.
	Moreover,~$\zeta_k \in E_{n_k}(\seq{w}) \subseteq E(\seq{w})$ as~$z_k \in E_{n_k}(\seq{w})$.
	This means that~$z$ is an accumulation point of~$E(\seq{w})\setminus\{z\}$, i.e.~$z \in E'(\seq{w})$.
\end{proof}

\begin{cor}
	Let~$\seq{p}$ be a sequence of monomials as in \cref{thm:Kuratowski_cand_derived}.
	Then, for any~$w\in \CC\setminus \{0,\infty\}$ we have~$\jul(\seq{p}) = \limsup_{n\to+\infty} \big(F_{\seq{p}}^{(n)}\big)^{-1}\{w\}$.
	Furthermore, for any non-empty finite set~$B\subseteq\CC\setminus\{0,\infty\}$ it holds that~$\jul(\seq{p}) = \limsup_{n\to+\infty} \big(F_{\seq{p}}^{(n)}\big)^{-1} B$.
\end{cor}

\begin{proof}
	The first claim follows directly from the theorem above.
	
	Now, take~$B$ to be any finite non-empty subset of~$\CC\setminus\{0,\infty\}$.
	For~$b\in B$, set
	\[
		E_n(b) := \big(F_{\seq{p}}^{(n)}\big)^{-1}\{b\} \inline{and} E_n(B) := \big(F_{\seq{p}}^{(n)}\big)^{-1} B.
	\]
	Immediately, we have that
	\[
		\jul(\seq{p}) = \limsup\nolimits_{n\to+\infty}E_n(b)
	\implies
		\jul(\seq{p}) \subseteq \limsup\nolimits_{n\to+\infty}E_n(B).
	\]
	Conversely take~$z\in \limsup_{n\to+\infty}E_n(B)$.
	Then, there exist indices~$n_k\to+\infty$ and points~$z_k\in E_{n_k}(B)$ such that~$z_k \to_\sigma z$,
	and, in turn, for each~$k$ there is a~$b_k \in B$ such that~$F_{\seq{p}}^{(n_k)}(z_k)=b_k$.
	But~$B$ is finite,
	and therefore there is a subsequence~$(b_{m_k})_{k\geq1}$ of~$(b_k)$ such that~$b_{m_k}=b$ for all~$k\geq1$ for some fixed~$b \in B$.
	Hence,~$z_{m_k} \in \big(F_{\seq{p}}^{(n_{m_k})}\big)^{-1}\{b\} = E_{n_{m_k}}(b)$ for each~$k\geq1$,
	and thus~$z \in \limsup_{n\to+\infty}E_n(b) = \jul(\seq{p})$,
	since~$z_{m_k} \to_\sigma z$.
\end{proof}


\section{Consequences: a factory of examples}
\label{sec:examples and contructions}


\Cref{thm:monomial_julia_accumulation} identifies the geometry of~$\jul(\seq{p})$ for a monomial sequence~$\seq{p}$ in terms of the accumulation set~$S\subseteq\overline{\mbR}$ of the real sequence~$(s_n)_{n\geq1}$.
In fact, the non-autonomous Julia sets for a sequence of monomials are unions of circles centred at~$0$ and with radii in the closed set~$E=\{e^{-s} \such s\in S\} \subseteq[0,+\infty]$ (where~$C(0,+\infty) := \{\infty\}$).
Moreover, for any non-empty closed set~$E\subseteq[0,+\infty]$ one can choose the coefficients~$a_n$ and degrees~$d_n$ such that~$E$ is realized.
In other words, in order to \emph{build} a Julia set with prescribed ``radial geometry'' it suffices to prescribe the set~$S$ of accumulation points of~$s_n = \frac{1}{D_n} \log|A_n|$,
and then choose appropriate coefficients~$a_n$ and degrees~$d_n$ for~$\seq{p}$ realising this sequence.

In this section, we first explain how one can derive a sequence~$(s_n)_{n\geq1}$ of points in~$\mbR$,
whose set~$S\subseteq\overline{\mbR}$ of accumulation points is given,
and subsequently a sequence of monomials realising the resulting sequence~$(s_n)$.
Then, we proceed to extract a wide range of constructions and examples as direct consequences of \cref{thm:monomial_julia_accumulation} and the invariance principle \cref{thm:invariance deterministic}.

\begin{proof}[Construction of a monomial sequence from given set of accumulation points]
	Let~$S$ be a non-empty closed subset of~$\overline{\mbR}$.
	
	First, we consider the case when~$\pm\infty \not\in S$.
	Let~$m\in\mbN$ and consider the dyadic intervals
	\[
		I_{m,k} = [k2^{-m}, (k+1)2^{-m}] \inline{so that} \bigcup\nolimits_{k=-m2^m}^{m2^m-1} I_{m,k} = [-m,m];
	\]
	there are~$2m2^m$ many such intervals~$I_{m,k}$ for each~$m$.
	Notice that the intersection~$S\cap I_{m,k}$ is always a compact set (for all possible~$m$ and~$k$).
	For each~$m\in\mbN$,
	if~$S\cap I_{m,k} = \emptyset$ for all~$k \in [-m2^m,m2^m)\cap\mbZ$,
	we set~$S_m := \emptyset$;
	otherwise, we set
	\[
		x_{m,k} := \min \bigl(S\cap I_{m,k}\bigr) 
	\inline{and}
		S_m := \bigcup\nolimits_{k=-m2^m}^{m2^m-1} \{x_{m,k} \such \text{if}\ S\cap I_{m,k} \neq \emptyset\}.
	\]
	We define the sequence~$(s_n)_{n\in\mbN}$ by first listing all the elements of~$S_1$ (as~$s_1$,~$s_2$,~$s_3$, etc.\ until~$S_1$ is exhausted), then the elements of~$S_2$, then of~$S_3$, and so on.
	
	To see that~$S$ is precisely the accumulation points of~$(s_n)$,
	observe first that~$s_n \in S$ for all~$n\in\mbN$ by construction.
	Since~$S$ is closed (in~$\overline{\mbR}$),
	every subsequential limit of~$(s_n)$ belongs to~$S$.
	Conversely, let~$x \in S$.
	Necessarily,~$x\in[-m,m]$ for some sufficiently large~$m$,
	and choose~$k_m$ so that~$x \in I_{m,k_m}$.
	By construction,~$x_{m,k_m} \in I_{m,k_m}$,
	and therefore~$|x - x_{m,k_m}| \leq 1/2^m$.
	Iterating for all larger~$m$, we get a subsequence of~$(s_n)$ converging to~$x$.
	
	Second, assume that~$+\infty$, or~$-\infty$, or both are elements of~$S$.
	We repeat the same construction for~$(s_n)$ as above with the following addition:
	For each~$m\in\mbN$, after we have exhausted the points of~$S_m$,
	we append the point(s)~$\pm m$ accordingly in the numbering of~$(s_n)$.
	In this case, the sequence(s)~$(\pm1,\pm2,\pm3,\dots)$ is/are a subsequence of~$(s_n)$ which converge(s) to~$\pm\infty$.
	For the set~$S\setminus\{\pm\infty\}$, the argument above readily applies.
	
	To construct the appropriate monomials yielding~$(s_n)$, notice \eqref{eq:s_n} immediately implies
	\[
		|a_n| = e^{D_n(s_n-s_{n-1})} \afterline{for all} n\geq2 \inline{and also} |a_1|=e^{d_1s_1}.
	\]
	In picking the~$d_i$'s and the~$a_i$'s, one needs only make sure that the above equalities are satisfied
	and that infinitely many of the degrees~$d_i$'s are at least~$2$.
	In particular, one can pick
	\[
		d_n=2,
		\inline{and}
		a_n = e^{2^n(s_n-s_{n-1})} \afterline{for all} n\geq1
	\]
	with~$s_0=0$,
	which gives a sequence of quadratic monomials.
	This will be a frequent choice in our examples below.
\end{proof}

The above argument immediately yields the following additional property.

\begin{prop}
	Let~$\seq{p}=(p_n)_{n\ge1}$ be a sequence of monomials~$p_n(z)=a_n z^{d_n}$ with~$a_n\in\mbC\setminus\{0\}$,~$d_n\ge1$ for all~$n\in\mbN$.
	Let~$\seq{\theta} = (\theta_1,\theta_2,\dots) \in [0,2\pi)^\mbN$ be a sequence of angles.
	Then, for the sequence of monomials~$e^{i\seq{\theta}} \odot \seq{p} := (e^{i\theta_1}p_1,e^{i\theta_2}p_2,e^{i\theta_3}p_3,\dots)$ it holds that~$\jul(e^{i\seq{\theta}} \odot \seq{p}) = \jul(\seq{p})$,
	and~$\fat(e^{i\seq{\theta}} \odot \seq{p}) = \fat(\seq{p})$.
\end{prop}


\subsection{Level-set constructions from monomial sequences}
\label{sec:monomial_julia_trasport}


In this section, we turn \cref{thm:monomial_julia_accumulation} into a flexible construction scheme.
By inserting a single non-constant rational map~$f$ at the beginning of the itinerary~$(p_1,p_2,p_3,\dots)$ and using the invariance principle of \cref{thm:invariance deterministic},
the circles making up the Julia sets pull back to \emph{level sets} of~$|f|$.
Here, we streamline this idea to produce a plethora of examples.

For convenience, we write
\[
	|f|^{-1}\{\rho\} := \{z\in\CC \such |f(z)|=\rho\},
	\inline{and}
	|f|^{-1}E := \{z\in\CC \such |f(z)| \in E\}.
\]
for any~$\rho\in[0,+\infty]$ and any set~$E\subseteq[0,+\infty]$.

\begin{thm}
\label{thm:pullback_levelsets}
	Let~$f:\CC\to\CC$ be a non-constant rational map and~$E\subseteq[0,+\infty]$ a non-empty closed set in~$\overline{\mbR}$.
	There exists a sequence of quadratic monomials~$\seq{p}=(p_1,p_2,\dots)$ such that
	\[
		\jul(\tilde{\seq{p}})=|f|^{-1}E,
	\]
	where the rational sequence~$\tilde{\seq{p}}$ is given by~$\tilde{\seq{p}} := (f,p_1,p_2,\dots)$.
\end{thm}

\begin{proof}
	Consider the non-empty closed set~$S = \{-\log r \such r \in E\} \subseteq \overline{\mbR}$.
	It is (always) possible to find a sequence~$(s_n)_{n\in\mbN}$ in~$\mbR$,
	whose set of accumulation points is exactly~$S$.
	Set~$s_0=0$,
	and consider the monomials~$p_n(z) = e^{2^n(s_n-s_{n-1})} z^2$ for all~$n\geq1$.
	\Cref{thm:monomial_julia_accumulation}, then, yields that
	\[
		\jul(\seq{p}) = \bigcup\nolimits_{s\in S}\{z\in\CC \such |z|=e^{-s}\} = \{z\in\CC \such |z|\in E\}.
	\]
	
	Now, consider the rational sequence~$\tilde{\seq{p}}=(f,p_1,p_2,\dots)$.
	Our claim follows using \cref{thm:invariance deterministic}:
	\[
		\jul(\tilde{\seq{p}})
		= f^{-1} \jul(\shift\tilde{\seq{p}})
		= f^{-1} \jul(\seq{p})
		= f^{-1} \{z\in\CC \such |z|\in E\}
		= \{w\in\CC \such |f(w)|\in E\}
		= |f|^{-1}E.
\qedhere
	\]
\end{proof}

We remark that the monomials in \cref{thm:pullback_levelsets} need not be quadratic in general.

\Cref{thm:pullback_levelsets} shows that the geometry described in \cref{thm:monomial_julia_accumulation} can be \emph{transported} far beyond concentric circles:
for any non-constant rational map~$f$ and any non-empty closed~$E\subseteq[0,+\infty]$, the set~$|f|^{-1}E$ occurs as a non-autonomous Julia set.
When~$f$ happens to be a polynomial, the level curves~$|f|^{-1}\{\rho\}$ for~$\rho\in[0,+\infty]$ are (polynomial) \emph{lemniscates};
in particular when~$f(z)=(z-a)(z-b)$, they are the classical \emph{Cassini ovals},
and when~$f(z)=z^d-1$ for some~$d\in\mbN$ they form~$d$-fold symmetric ``petalled'' lemniscates/roses.
When~$f$ is a M{\"o}bius transformation, the level curves are generalized circles;
for instance,~$f(z)=\frac{z-1}{z+1}$ produces an \emph{Apollonius pencil} of coaxial generalised circles.
See these examples in \cref{fig:levelset_gallery}.

\begin{figure}[!ht]
\centering
	\begin{subfigure}[t]{0.48\linewidth}
	\centering
		\includegraphics[width=\linewidth]{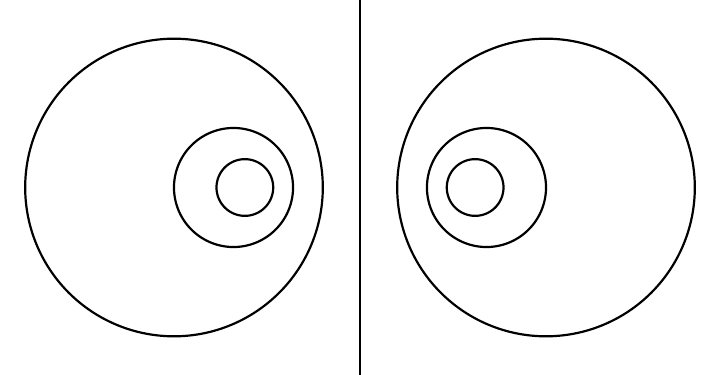}
	\caption{An Apollonius pencil of coaxial generalized circles:~$f(z)=\frac{z-1}{z+1}$.}
	\label{fig:levelset_pencil}
	\end{subfigure}
\hfill
	\begin{subfigure}[t]{0.48\linewidth}
	\centering
		\includegraphics[width=\linewidth]{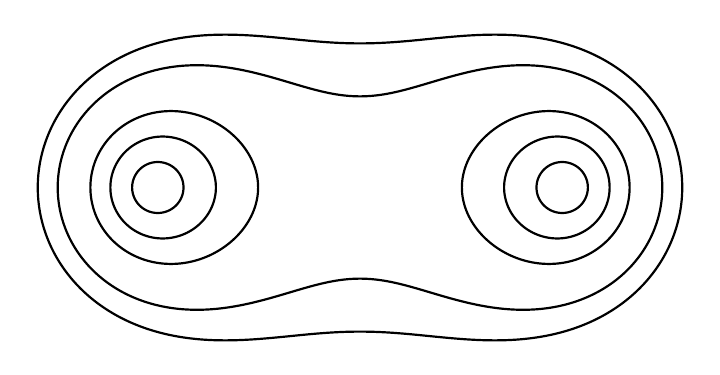}
	\caption{Cassini-type lemniscates:~$f(z)=(z-a)(z-b)$.}
	\label{fig:levelset_cassini}
	\end{subfigure}
	
	\begin{subfigure}[t]{0.32\linewidth}
	\centering
		\includegraphics[width=\linewidth]{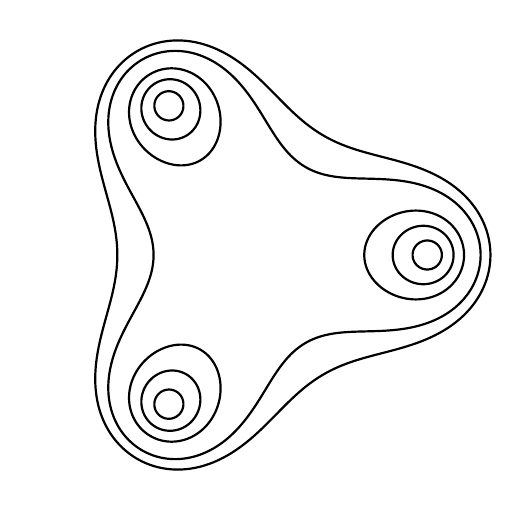}
	\caption{Lemniscate family ($3$-fold symmetry):~$f(z)=z^3-1$.}
	\label{fig:levelset_m3}
	\end{subfigure}
\hfill
	\begin{subfigure}[t]{0.32\linewidth}
	\centering
		\includegraphics[width=\linewidth]{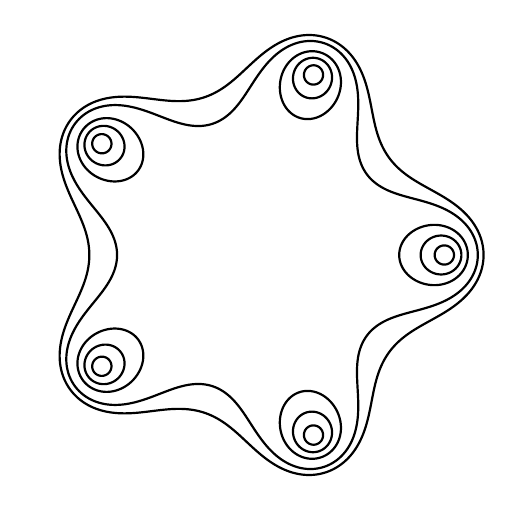}
	\caption{Lemniscate family ($5$-fold symmetry):~$f(z)=z^5-1$.}
	\label{fig:levelset_m5}
	\end{subfigure}
\hfill
	\begin{subfigure}[t]{0.32\linewidth}
	\centering
		\includegraphics[width=\linewidth]{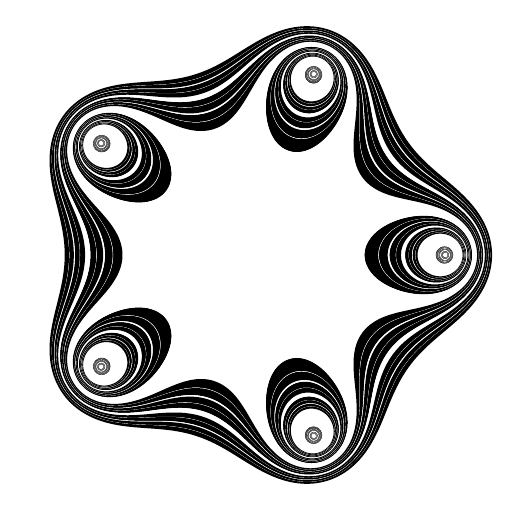}
	\caption{``Thick'' lemniscate band ($5$-fold symmetry):~$f(z)=z^5-1$.}
	\label{fig:levelset_m5_thick}
	\end{subfigure}
\caption{%
	Examples of Julia sets produced through \cref{thm:pullback_levelsets}:
	each panel depicts the set~$|f|^{-1}E$ for a rational map~$f$ and a closed set~$E\subseteq[0,+\infty]$.
	}
\label{fig:levelset_gallery}
\end{figure}

An interesting question is what kinds of examples of Julia sets can be produced through \cref{thm:pullback_levelsets}.
One can study how the geometry and the ``level of disconnectedness'' of the set~$E\subseteq\overline{\mbR}$ determines the corresponding properties of the set~$|f|^{-1}E$ for a given rational map~$f$, or for a given closed set~$E$.
This is beyond the scope of this text,
but we do provide some interesting cases in the sections to follow.


\subsection{Julia set with empty interior and positive area}
\label{sec:empty int and positive area}


A Julia set in the autonomous setting has necessarily no interior unless it is the entire Riemann sphere.
However, there are examples of autonomous Julia sets (which are not the sphere) with positive area \cite{Buff2012}.
In the non-autonomous setting, the existence of such Julia sets with no interior yet positive area becomes immediate thanks to \cref{thm:monomial_julia_accumulation}.
We construct such a Julia set below in \cref{eg:jul_no_int_pos_area},
but first some basic properties are due.

We denote by~$m_1$ the~$1$-dimensional Lebesgue measure on~$\mbR$ and by~$m_2$ the~$2$-dimensional Lebesgue measure on~$\mbR^2 \simeq \mbC$.

\begin{lem}
\label{lem:interior_area}
	Let~$E\subseteq[0,+\infty)$ be a Borel set and let~$J_E = \{z\in\mbC \such |z| \in E\}$.
	Then,
	\[
		m_2(J_E) = 2\pi\int_E t\,dm_1(t),
	\]
	and the following hold:
	\begin{enumerate}
		\item $J_E$ has interior if, and only if,~$E$ contains an open interval;
		\item if~$E\subseteq[0,R]$ for some~$R\in[0,+\infty)$, then~$m_2(J_E)>0$ if, and only if,~$m_1(E)>0$.
	\end{enumerate}
\end{lem}

\begin{proof}
	If~$(a,b) \subseteq E$,
	we trivially have that~$J_E$ contains the annulus~$\{z\in\CC \such a<|z|<b\}$ and thus has interior.
	(The value of~$b$ can be~$+\infty$ here.)
	Conversely, if~$J_E$ has interior,
	then there is some~$z_0\in J_E$ and~$r>0$ such that~$D(z_0,r) \subseteq J_E$.
	If~$z_0=0$, then~$(0,r)\subseteq E$.
	If~$z_0\neq0$, choose~$0 < \varepsilon < \min\{r,|z_0|\}$.
	Then,~$(|z_0|-\varepsilon, |z_0|+\varepsilon) \subseteq E$.
	This proves (1).
	
	Since~$E$ is Borel-measurable and the map~$z\mapsto|z|$ is continuous,~$J_E=|\cdot|^{-1}(E)$ is~$m_2$-measurable.
	Using polar coordinates for any~$m_2$-measurable set~$A\subseteq\mbC\simeq\mbR^2$,
	\[
		m_2(A)=\int_0^{2\pi}\int_0^{+\infty} \mathbf{1}_A \big(te^{i\theta}\big)t \,dt \,d\theta.
	\]
	In particular for~$A=J_E$, we have~$\mathbf{1}_{J_E}(te^{i\theta})=\mathbf{1}_E(t)$, which gives
	\[
		m_2(J_E)=\int_0^{2\pi}\int_0^{+\infty} \mathbf{1}_E(t) t \,dt \,d\theta
		=2\pi\int_E t\,dm_1(t).
	\]
	
	If, in addition,~$E\subseteq[0,R]$ for some~$R\in[0,+\infty)$, then~$t \le R$ for~$t \in E$.
	From the above equality, this implies that~$m_2(J_E)\le 2\pi R\,m_1(E)$,
	and consequently if~$m_2(J_E)>0$, then~$m_1(E)>0$.
	On the other hand,~$t>0$ on~$(0,R)$ (for~$R>0$) and if~$m_1(E)>0$, then~$m_2(J_E)=2\pi\int_E t\,dm_1(t)>0$.
\end{proof}

\Cref{lem:interior_area} combined with \cref{thm:monomial_julia_accumulation} tells us a Julia set of a monomial sequence has empty interior if, and only if,~$E$ does not contain an interval,
which in turn happens if, and only if,~$S$ does not contain an interval.
Moreover, if~$S\cap\mbR \subseteq [a,b]$ for some~$a,b\in\mbR$,
then the Julia set has positive area if, and only if,~$m_1(S\cap\mbR) > 0$.
Therefore, the construction of a sequence of monomials for which the Julia set has empty interior but positive area becomes equivalent to that of the construction of a Smith-Volterra-Cantor set, or a fat Cantor set.

We note that in the autonomous setting Cantor-circles Julia sets have already been constructed \cite{McMullen1988,QIU2013},
but the resulting Julia sets have planar Lebesgue measure zero.
In the literature, a \emph{Cantor set of circles} (or \emph{Cantor circles}) refers to a set consisting of uncountably many Jordan curves and homeomorphic to~$\mcC\times\mbT$,
where~$\mcC$ is a Cantor set and~$\mbT$ the unit circle.

\begin{eg}
\label{eg:jul_no_int_pos_area}
	Let~$\mcC\subseteq[0,1]$ be the fat Cantor set obtained by removing the middle open interval of length~$4^{-k}$ from each (closed) interval in the finite union at stage~$k$.
	It is known that~$\mcC$ has a positive~$m_1$-measure and no interior.
	
	Next, set~$S = -\log\mcC = \{-\log c \such c\in\mcC\} \subseteq \overline{\mbR}$ and~$J_\mcC=\{z\in\mbC \such |z|\in\mcC\}$.
	We can find a sequence~$(s_n)_{n\in\mbN}$,
	whose accumulation points are exactly the set~$S$,
	and thus a sequence of monomials~$\seq{p}$
	so that
	\[
		\jul(\seq{p}) = \bigcup\nolimits_{s \in S}\{z\in\CC \such |z|=e^{-s}\} = J_\mcC
	\]
	owing to \cref{thm:monomial_julia_accumulation}.
	According to \cref{lem:interior_area},~$J_\mcC$ has also no interior but positive~$m_2$-measure,
	and therefore we have constructed a Julia set of a monomial sequence with these properties.
\qed
\end{eg}

The construction of the sequence~$(s_n)_{n\in\mbN}$ above can be done following the construction of the fat Cantor set itself instead of the method described at the beginning of \cref{sec:examples and contructions}.

It is also worth noting that \cite[Example~1.19]{Stankewitz2011} is the Julia set of a \emph{polynomial semigroup} generated by two monomials, which is a Cantor set of circles.
(See \cite{Hinkkanen1996} for background on rational semigroups).
In that setting, the Fatou set is defined as the set of normality of the family consisting of all finite compositions of the generators in all possible orders.
By contrast, our construction in \cref{eg:jul_no_int_pos_area} uses a \emph{single prescribed non-autonomous sequence} of monomials.

This fits naturally into a fiber-wise or a skew-product viewpoint.
Let~$\Gamma := \{p_n \such n\ge 1\}~$ be the (countable) collection of monomials that occur in \cref{eg:jul_no_int_pos_area},
and set~$\Sigma := \Gamma^{\mbN}$.
Each~$\seq{x}=(\gamma_1,\gamma_2,\dots)\in\Sigma$ is a \emph{base point} for the associated skew product,
and the copy~$\{\seq{x}\}\times\CC$ is the \emph{fiber over~$\seq{x}$}.
The point~$\seq{x}$ determines a sequence of non-autonomous composition iterates~$ F_\seq{x}^{(n)} := \gamma_n\circ\dots\circ\gamma_1$ for~$n\ge 1$, 
and we can define the \emph{fiber-wise Julia set} over~$\seq{x}$ by
\[
	\jul_\seq{x}
	:= \big\{ z\in\CC \such \{F_\seq{x}^{(n)}\}_{n\ge 1} \inmath{is not normal in any neighbourhood of} z \big\}.
\]
With this notation, our original sequence~$\seq{p}=(p_1,p_2,\dots)$ corresponds to the base point~$\seq{x}_0=(p_1,p_2,\dots)\in\Sigma$ so that the fiber over~$\seq{p}$ is~$\{\seq{x}_0\}\times\CC$.
The associated non-autonomous iterates are~$F^{(n)}_{\seq{x}_0} = p_n\circ\dots\circ p_1$,
hence our non-autonomous Julia set coincides with the fiber-wise Julia set over~$\seq{x}_0$,
that is~$\jul(\seq{p})=\jul_{\seq{x}_0}$.
This is the same fiber-wise notion of a Julia set along a sequence defined in \cite{Stankewitz2011},
where one considers sequences of choices from a finite generating set.


\subsection{Non-uniformly perfect Julia sets}
\label{sec:NUP Julia}


It is known that in the autonomous case for any rational map with degree at least~$2$ the Julia set is infinite and perfect.
Already from \cref{thm:monomial_julia_accumulation}, it is evident that for monomial sequences this need not be the case.
In fact, one can find monomial sequences for which any of the following can hold:
\begin{enumerate}
	\item The Julia set has non-empty interior without being the entire sphere~$\CC$.
	\item The Julia set has~$0$ and/or~$\infty$ as isolated point(s) and also contains a circle or an annulus;
	hence it is not perfect, with or without non-empty interior
	(\cref{eg:not_perfect_not_UP}).
	\item The Julia set can be one of the degenerate cases~$\{0\}$,~$\{\infty\}$ or~$\{0,\infty\}$
	(\cref{thm:any_finite_subset}).
\end{enumerate}

The points~$0$ and~$\infty$ are the only possible isolated points a Julia set of a monomial sequence can have,
and these can be precisely identified using \cref{cor:monomial_isolated_points}.
In fact, this corollary gives a concrete criterion for whether the Julia set of a monomial sequence is perfect or not.

However, there is a stronger notion of perfectness, \emph{uniform perfectness}, introduced by Pommerenke in \cite{Pommerenke1979}.
A \emph{conformal annulus}~$A$ is a doubly connected domain in~$\mbC$,
which is conformally equivalent to a round annulus
\[
	A(z_0;r,R) = \{z\in\mbC \such r<|z-z_0|<R\} \afterline{for some} z_0\in\mbC \inmath{and} 0 \leq r < R < +\infty.
\]
By the Riemann mapping theorem for doubly connected domains, the radii~$r$ and~$R$ are uniquely determined up to a multiplicative factor,
i.e.\ the ratio~$R/r$ is uniquely determined by~$A$.
The \emph{modulus} of a conformal annulus~$A$ is defined as
\[
	\mod A := \log(R / r),
\]
whenever~$A$ conformally equivalent to~$A(z_0;r,R)$.
The quantity~$\mod A$ is, in fact, conformally invariant:
if~$\varphi: A \to B$ is a conformal map, then~$\mod A = \mod B$.
We note that some authors include a factor of~$\frac{1}{2\pi}$ in the definition of the modulus.

\begin{dfn}
\label{dfn:separating_annulus and UP}
	A conformal annulus~$A$ is said to \emph{separate} a set~$F\subseteq\mbC$ whenever~$F\cap A=\emptyset$ and~$F$ intersects both (connected) components of~$\mbC \setminus A$.
	
	A compact set~$F\subseteq\mbC$ with at least two points is called \emph{uniformly perfect} whenever the collection of all conformal annuli which separate it has uniformly bounded modulus.
\end{dfn}

Uniform perfectness is stronger than perfectness
in the sense that for a compact set~$F\subseteq\mbC$ with at least~$2$ points,
if~$F$ is uniformly perfect, then it is also perfect.
Indeed, if~$z_0\in F$ is isolated in~$F$,
then~$\delta := \dist(z_0, F\setminus\{z_0\})$ is necessarily positive.
However, \emph{for any}~$r \in (0,\delta)$ the annulus~$A(z_0;r,\delta)$ separates~$F$,
which is a contradiction as~$\mod A(z_0; r,\delta) = \log (\delta/r)\to+\infty$ as~$r\to0$.
Obviously, a non-perfect set cannot be uniformly perfect.

It is known \cite{Hinkkanen,ManeDaRocha} that for autonomous iterations of rational functions of degree at least~$2$ the Julia sets are necessarily uniformly perfect.
In the non-autonomous setting, ``bounded sequences of polynomials'' (defined, for example, in \cite{ComerfordSumi-ThickThin})
also yield uniformly perfect Julia sets \cite{SumiFiberedRational2006}.
Thus, to obtain a non-uniformly perfect Julia set, one must look into Julia sets of sequences of polynomials, which are not ``bounded polynomials''.
With the help of \cref{thm:monomial_julia_accumulation}, one can indeed find such a sequence of \emph{monomials}.
We also note that any connected set is automatically uniformly perfect,
since no annulus can separate a connected set.
So any non-uniformly perfect Julia set must necessarily be disconnected.
Below, we provide two relevant examples.

\begin{eg}
\label{eg:not_up_but_perfect}
	There exists a sequence of monomials,
	whose Julia set is perfect but not uniformly perfect.
\end{eg}

\begin{proof}
	Let~$b>1$,
	and consider the set~$S=\{n^2 \log b \such n\in\mbN\}\cup\{+\infty\}$.
	As explained at the beginning of \cref{sec:examples and contructions},
	we can always find a sequence of monomials~$\seq{p}$,
	which produces a sequence~$(s_n)_{n\geq1}$ with accumulation points exactly the set~$S$.
	Particularly, we can take 
	\[
		p_n(z) = b^{2^n(t_n-t_{n-1})}z^2
	\inline{where}
		t_0 = 0 \inmath{and} t_{k(k-1)/2 + j} = j^2 \inmath{for} 1 \le j \le k
	\]
	so that~$s_n = t_n \log b$.
	On the other hand, this set~$S$ gives rise to the Julia set
	\[
		\jul(\seq{p}) = \{0\} \cup \bigcup\nolimits_{n\geq1} C(0,b^{-n^2})
	\]
	owing to \cref{thm:monomial_julia_accumulation}.
	Clearly,~$\jul(\seq{p})$ is perfect.
	To see it is not uniformly perfect,
	consider the annuli~$A(0;b^{-(n+1)^2},b^{-n^2})$,
	which separate~$\jul(\seq{p})$ for all~$n\geq1$,
	but have modulus
	\[
		\mod A = \log\frac{b^{-n^2}}{b^{-(n+1)^2}} = (2n+1)\log b \to +\infty \afterline{as} n\to+\infty.
\qedhere
	\]
\end{proof}

\begin{eg}
\label{eg:not_perfect_not_UP}
	There exists a sequence of monomials,
	whose Julia set is not perfect.
\end{eg}

\begin{proof}
	Consider the sequence~$(s_n)_{n\geq0}$ defined by
	\[
		s_0=0,	\quad	s_{2k}=0, \inline{and} s_{2k-1}=k \afterline{for} k\geq1
	\]
	with points of accumulation the set~$S=\{0,+\infty\}$.
	For an appropriate sequence of monomials~$\seq{p}$, we get the (non-perfect) Julia set~$\jul(\seq{p}) = \{0\} \cup C(0,1)$ through \cref{thm:monomial_julia_accumulation}.
\end{proof}


\subsection{Finite Julia sets}
\label{sec:finite Julia}


In contrast to the autonomous iteration of a single rational map of degree at least~$2$, finite Julia sets do occur in the non-autonomous setting.
In fact, every finite subset of the Riemann sphere can be realised already by a sequence of \emph{polynomials}.
After we had produced our examples to showcase this fact,
we realised similar ones already existed in \cite{Bue1995}.
We present them here nonetheless to further showcase the versatility of \cref{thm:monomial_julia_accumulation} when combined with the invariance principle \cref{thm:invariance deterministic}.

\begin{thm}
\label{thm:any_finite_subset}
	Every finite subset of~$\CC$ is the Julia set of some sequence of polynomials.
	In particular, finite Julia sets of arbitrary cardinality can occur.
\end{thm}

\begin{proof}
	We first construct a tentative monomial sequence,
	on which we will be applying \cref{thm:monomial_julia_accumulation}.
	We set~$d_n=2$ for all~$n\in\mbN$,~$s_0 := 0$, and for a prescribed real sequence~$(s_n)_{n\ge1}$
	\[
		p_n(z) := a_n z^2,
		\inline{with}
		a_n := e^{2^n(s_n-s_{n-1})}
		\afterline{for all} n\geq1.
	\]
	Then,~$D_n=2^n$ and~$s_n = \frac{1}{2^n}\log|A_n|=\sum_{i=1}^n \frac{1}{2^i} \log a_i$ for all~$n\ge1$,
	and let~$S\subseteq\overline{\mbR}$ be the set of accumulation points of~$(s_n)$.
	
	Now, let~$A$ be a finite subset of the Riemann sphere~$\CC$.
	
	If~$A=\varnothing$,
	take~$q_n(z)=z$ for all~$n\geq1$;
	then~$\jul(\seq{q})=\varnothing$.
	
	If~$A=\{0\}$,
	take~$a_n = e^{2^n}$ so that~$s_n=n$ and~$S=\{+\infty\}$.
	Then,~$\jul(\seq{p}) = \{0\} = A$.
	
	If~$A=\{\infty\}$,
	take~$a_n=e^{-2^n}$ so that~$s_n=-n$ and~$S=\{-\infty\}$.
	Then,~$\jul(\seq{p}) = \{\infty\} = A$.
	
	If~$A=\{0,\infty\}$,
	take
	\begin{equation}
	\label{eq:coefficients for jul 0 and infty}
		a_n =
			\begin{cases}
				e^{-n2^n}	&	\text{when}\ n\ \text{is even}
			\\	e^{n2^n}	&	\text{when}\ n\ \text{is odd}
			\end{cases}
		\inline{so that}
		s_n =
			\begin{cases}
				-\frac{n}{2}	&	\text{when}\ n\ \text{is even}
			\\	\frac{n+1}{2}	&	\text{when}\ n\ \text{is odd}
			\end{cases}
	\end{equation}
	and~$S=\{-\infty,+\infty\}$.
	Then,~$\jul(\seq{p}) = \{0,\infty\}$.
	
	If~$A=\{z_1,z_2,\dots,z_k\} \subseteq \CC\setminus\{\infty\}$ for some~$k\in\mbN$,
	take~$a_n=e^{2^n}$ and~$f(z)=(z-z_1)(z-z_2)\dots(z-z_k)$,
	and consider the sequence~$\tilde{\seq{p}} = (f,p_1,p_2,\dots)$.
	Then,
	\[
		\jul(\tilde{\seq{p}}) = f^{-1} \jul(\seq{p}) = f^{-1} \{0\} = \{z_1,z_2,\dots,z_k\} = A.
	\]
	
	If~$A=\{z_1,z_2,\dots,z_k,\infty\}$ for some~$k\in\mbN$,
	take~$a_n$ as in \eqref{eq:coefficients for jul 0 and infty} and~$f(z)=(z-z_1)(z-z_2)\dots(z-z_k)$,
	and consider the sequence~$\tilde{\seq{p}} = (f,p_1,p_2,\dots)$.
	Then,
	\[
		\jul(\tilde{\seq{p}}) = f^{-1} \jul(\seq{p}) = f^{-1} \{0,\infty\} = \{z_1,z_2,\dots,z_k,\infty\} = A.
\qedhere
	\]
\end{proof}


\section*{Acknowledgements}


LP was supported by Villum Fonden Grant No.~25452 and Grant No.~60842, QMATH Center of Excellence Grant No.~10059 as well as Danish e-infrastructure Consortium (DeiC) Grant No.~5260-00014B.

CLP acknowledges the support from Danmarks Frie Forskningsfond Grant No.~1026-00267B.

DV was supported by the Alexander von Humboldt-Stiftung through the Humboldt Professorship awarded to Stefanie Petermichl.


\printbibliography


\end{document}